\documentclass[12pt, reqno]{amsart}

\usepackage{amsmath,amssymb,amscd,mathrsfs,amscd, mdwlist}
\usepackage{graphics,verbatim, tikz}
\allowdisplaybreaks

\usetikzlibrary{matrix,arrows,patterns, snakes}
\usepackage{amsmath,amstext,amsbsy,amssymb,amscd}
\usepackage{amsmath}
\usepackage{amsxtra}
\usepackage{amscd}
\usepackage{amsthm}
\usepackage{amsfonts}
\usepackage{graphicx}%
\usepackage{amssymb}
\usepackage{eucal}
\usepackage{color}
\usepackage{multirow}
\usepackage[enableskew,vcentermath]{youngtab}

\newtheorem{thm}{Theorem}[section]
\theoremstyle{plain}
\newtheorem*{thm*}{{\bf Theorem}}
\newtheorem{lem}[thm]{Lemma}
\newtheorem{prop}[thm]{Proposition}

\newtheorem{cor}[thm]{Corollary}

\theoremstyle{definition}
\newtheorem{dfn}[thm]{Definition}

\theoremstyle{remark}

\definecolor{A}{rgb}{.75,1,.75}

\numberwithin{equation}{section}

\newcommand{\cR}{{\mathcal{R}}}

\newcommand{\height}{{\mathrm{ht}}}

\renewcommand{\bar}[1]{\overline{#1}}

\DeclareMathOperator{\qtr}{{\rm qtr}}

\DeclareMathOperator{\coqtr}{{\rm coqtr}}

\newcommand{\undercurvearrowleft}{\raisebox{.5em}{\rotatebox{180}{$\curvearrowright$}}}
\newcommand{\undercurvearrowright}{\raisebox{.5em}{\rotatebox{180}{$\curvearrowleft$}}}

\usepackage[all,knot]{xy}
\xyoption{arc}

\newcommand{\hctikz}[1]{\raisebox{-.5\height}{\begin{tikzpicture}#1\end{tikzpicture}}}

\usetikzlibrary{decorations.markings}
\tikzset{->->-/.style={decoration={
  markings,
  mark=at position .25 with {\arrow[scale=1.5,black]{>}}, mark=at position .85 with {\arrow[scale=1.5,black]{>}}},postaction={decorate}}}
\tikzset{->-/.style={decoration={
  markings,
  mark=at position .55 with {\arrow[scale=1.5,black]{>}}},postaction={decorate}}}
\tikzset{->/.style={decoration={
  markings,
  mark=at position .99 with {\arrow[scale=1.5,black]{>}}},postaction={decorate}}}
\tikzset{<-/.style={decoration={
  markings,
  mark=at position .10 with {\arrow[scale=1.5,black]{<}}},postaction={decorate}}}
\tikzset{-<-/.style={decoration={
  markings,
  mark=at position .55 with {\arrow[scale=1.5,black]{<}}},postaction={decorate}}}
\tikzset{-<-<-/.style={decoration={
  markings,
  mark=at position .25 with {\arrow[scale=1.5,black]{<}}, mark=at position .85 with {\arrow[scale=1.5,black]{<}}},postaction={decorate}}}

\newcommand{\nesc}[4]{
\pgfmathsetmacro{\xscale}{{#3}}
\pgfmathsetmacro{\yscale}{{#4}}
\pgfmathsetmacro{\xcenter}{{#1}}
\pgfmathsetmacro{\ycenter}{{#2}}
	\pgfmathparse{-.5*\xscale+\xcenter}		\let\Xa\pgfmathresult
    \pgfmathparse{-1*\yscale+\ycenter}		\let\Ya\pgfmathresult
    \coordinate (A) at (\Xa,\Ya);
	\pgfmathparse{-.5*\xscale+\xcenter}		\let\Xb\pgfmathresult
    \pgfmathparse{0+\ycenter}		\let\Yb\pgfmathresult
    \coordinate (B) at (\Xb,\Yb);
	\pgfmathparse{-.25*\xscale+\xcenter}		\let\Xc\pgfmathresult
    \pgfmathparse{.5*\yscale+\ycenter}		\let\Yc\pgfmathresult
    \coordinate (C) at (\Xc,\Yc);
	\pgfmathparse{0+\xcenter}		\let\Xd\pgfmathresult
    \pgfmathparse{0+\ycenter}		\let\Yd\pgfmathresult
    \coordinate (D) at (\Xd,\Yd);
	\pgfmathparse{.25*\xscale+\xcenter}		\let\Xe\pgfmathresult
    \pgfmathparse{-.5*\yscale+\ycenter}		\let\Ye\pgfmathresult
    \coordinate (E) at (\Xe,\Ye);
	\pgfmathparse{.5*\xscale+\xcenter}		\let\Xf\pgfmathresult
    \pgfmathparse{0*\yscale+\ycenter}		\let\Yf\pgfmathresult
    \coordinate (F) at (\Xf,\Yf);
	\pgfmathparse{.5*\xscale+\xcenter}		\let\Xg\pgfmathresult
    \pgfmathparse{1*\yscale+\ycenter}		\let\Yg\pgfmathresult
    \coordinate (G) at (\Xg,\Yg);
\draw[line width=1pt] (A) .. controls (B) and (C) .. (D) .. controls (E) and (F)  .. (G);
}
\newcommand{\dcap}[4]{
\pgfmathsetmacro{\xscale}{{#3}}
\pgfmathsetmacro{\yscale}{{#4}}
\pgfmathsetmacro{\xleft}{{#1}}
\pgfmathsetmacro{\ybottom}{{#2}}
	\pgfmathparse{\xleft}		\let\Xa\pgfmathresult
    \pgfmathparse{\ybottom}		\let\Ya\pgfmathresult
    \coordinate (A) at (\Xa,\Ya);
	\pgfmathparse{\xleft}		\let\Xb\pgfmathresult
    \pgfmathparse{\ybottom+\yscale}		\let\Yb\pgfmathresult
    \coordinate (B) at (\Xb,\Yb);
	\pgfmathparse{\xleft+\xscale}		\let\Xc\pgfmathresult
    \pgfmathparse{\ybottom+\yscale}		\let\Yc\pgfmathresult
    \coordinate (C) at (\Xc,\Yc);
	\pgfmathparse{\xleft+\xscale}		\let\Xd\pgfmathresult
    \pgfmathparse{\ybottom}		\let\Yd\pgfmathresult
    \coordinate (D) at (\Xd,\Yd);
\draw[line width=1pt] (A) .. controls (B) and (C) .. (D);
}
\newcommand{\cwcap}[4]{
\pgfmathsetmacro{\xscale}{{#3}}
\pgfmathsetmacro{\yscale}{{#4}}
\pgfmathsetmacro{\xleft}{{#1}}
\pgfmathsetmacro{\ybottom}{{#2}}
	\pgfmathparse{\xleft}		\let\Xa\pgfmathresult
    \pgfmathparse{\ybottom}		\let\Ya\pgfmathresult
    \coordinate (A) at (\Xa,\Ya);
	\pgfmathparse{\xleft}		\let\Xb\pgfmathresult
    \pgfmathparse{\ybottom+\yscale}		\let\Yb\pgfmathresult
    \coordinate (B) at (\Xb,\Yb);
	\pgfmathparse{\xleft+\xscale}		\let\Xc\pgfmathresult
    \pgfmathparse{\ybottom+\yscale}		\let\Yc\pgfmathresult
    \coordinate (C) at (\Xc,\Yc);
	\pgfmathparse{\xleft+\xscale}		\let\Xd\pgfmathresult
    \pgfmathparse{\ybottom}		\let\Yd\pgfmathresult
    \coordinate (D) at (\Xd,\Yd);
\draw[line width=1pt, ->-] (A) .. controls (B) and (C) .. (D);
}
\newcommand{\ccwcap}[4]{
\pgfmathsetmacro{\xscale}{{#3}}
\pgfmathsetmacro{\yscale}{{#4}}
\pgfmathsetmacro{\xleft}{{#1}}
\pgfmathsetmacro{\ybottom}{{#2}}
	\pgfmathparse{\xleft}		\let\Xa\pgfmathresult
    \pgfmathparse{\ybottom}		\let\Ya\pgfmathresult
    \coordinate (A) at (\Xa,\Ya);
	\pgfmathparse{\xleft}		\let\Xb\pgfmathresult
    \pgfmathparse{\ybottom+\yscale}		\let\Yb\pgfmathresult
    \coordinate (B) at (\Xb,\Yb);
	\pgfmathparse{\xleft+\xscale}		\let\Xc\pgfmathresult
    \pgfmathparse{\ybottom+\yscale}		\let\Yc\pgfmathresult
    \coordinate (C) at (\Xc,\Yc);
	\pgfmathparse{\xleft+\xscale}		\let\Xd\pgfmathresult
    \pgfmathparse{\ybottom}		\let\Yd\pgfmathresult
    \coordinate (D) at (\Xd,\Yd);
\draw[line width=1pt, -<-] (A) .. controls (B) and (C) .. (D);
}
\newcommand{\dcup}[4]{
\pgfmathsetmacro{\xscale}{{#3}}
\pgfmathsetmacro{\yscale}{{#4}}
\pgfmathsetmacro{\xleft}{{#1}}
\pgfmathsetmacro{\ybottom}{{#2}}
	\pgfmathparse{\xleft}		\let\Xa\pgfmathresult
    \pgfmathparse{\ybottom+\yscale}		\let\Ya\pgfmathresult
    \coordinate (A) at (\Xa,\Ya);
	\pgfmathparse{\xleft}		\let\Xb\pgfmathresult
    \pgfmathparse{\ybottom}		\let\Yb\pgfmathresult
    \coordinate (B) at (\Xb,\Yb);
	\pgfmathparse{\xleft+\xscale}		\let\Xc\pgfmathresult
    \pgfmathparse{\ybottom}		\let\Yc\pgfmathresult
    \coordinate (C) at (\Xc,\Yc);
	\pgfmathparse{\xleft+\xscale}		\let\Xd\pgfmathresult
    \pgfmathparse{\ybottom+\yscale}		\let\Yd\pgfmathresult
    \coordinate (D) at (\Xd,\Yd);
\draw[line width=1pt] (A) .. controls (B) and (C) .. (D);
}
\newcommand{\cwcup}[4]{
\pgfmathsetmacro{\xscale}{{#3}}
\pgfmathsetmacro{\yscale}{{#4}}
\pgfmathsetmacro{\xleft}{{#1}}
\pgfmathsetmacro{\ybottom}{{#2}}
	\pgfmathparse{\xleft}		\let\Xa\pgfmathresult
    \pgfmathparse{\ybottom+\yscale}		\let\Ya\pgfmathresult
    \coordinate (A) at (\Xa,\Ya);
	\pgfmathparse{\xleft}		\let\Xb\pgfmathresult
    \pgfmathparse{\ybottom}		\let\Yb\pgfmathresult
    \coordinate (B) at (\Xb,\Yb);
	\pgfmathparse{\xleft+\xscale}		\let\Xc\pgfmathresult
    \pgfmathparse{\ybottom}		\let\Yc\pgfmathresult
    \coordinate (C) at (\Xc,\Yc);
	\pgfmathparse{\xleft+\xscale}		\let\Xd\pgfmathresult
    \pgfmathparse{\ybottom+\yscale}		\let\Yd\pgfmathresult
    \coordinate (D) at (\Xd,\Yd);
\draw[line width=1pt, -<-] (A) .. controls (B) and (C) .. (D);
}
\newcommand{\ccwcup}[4]{
\pgfmathsetmacro{\xscale}{{#3}}
\pgfmathsetmacro{\yscale}{{#4}}
\pgfmathsetmacro{\xleft}{{#1}}
\pgfmathsetmacro{\ybottom}{{#2}}
	\pgfmathparse{\xleft}		\let\Xa\pgfmathresult
    \pgfmathparse{\ybottom+\yscale}		\let\Ya\pgfmathresult
    \coordinate (A) at (\Xa,\Ya);
	\pgfmathparse{\xleft}		\let\Xb\pgfmathresult
    \pgfmathparse{\ybottom}		\let\Yb\pgfmathresult
    \coordinate (B) at (\Xb,\Yb);
	\pgfmathparse{\xleft+\xscale}		\let\Xc\pgfmathresult
    \pgfmathparse{\ybottom}		\let\Yc\pgfmathresult
    \coordinate (C) at (\Xc,\Yc);
	\pgfmathparse{\xleft+\xscale}		\let\Xd\pgfmathresult
    \pgfmathparse{\ybottom+\yscale}		\let\Yd\pgfmathresult
    \coordinate (D) at (\Xd,\Yd);
\draw[line width=1pt, ->-] (A) .. controls (B) and (C) .. (D);
}

\newcommand{\ids}[5]{
\pgfmathsetmacro{\xl}{{#1}}
\pgfmathsetmacro{\yb}{{#2}}
\pgfmathsetmacro{\num}{{#5}}
\pgfmathsetmacro{\yheight}{{#4}}
\pgfmathsetmacro{\yt}{\yb+\yheight};
\pgfmathsetmacro{\xscale}{{#3}};
	\foreach \x in {1,...,\num}{
		\pgfmathparse{\xl+(\x-1)*\xscale}\let\xc\pgfmathresult;
		\draw[line width=1pt] (\xc,\yb) -- (\xc, \yt);
	}
}
\newcommand{\idsup}[5]{
\pgfmathsetmacro{\xl}{{#1}}
\pgfmathsetmacro{\yb}{{#2}}
\pgfmathsetmacro{\num}{{#5}}
\pgfmathsetmacro{\yheight}{{#4}}
\pgfmathsetmacro{\yt}{\yb+\yheight};
\pgfmathsetmacro{\xscale}{{#3}};
	\foreach \x in {1,...,\num}{
		\pgfmathparse{\xl+(\x-1)*\xscale}\let\xc\pgfmathresult
		\draw[line width=1pt,->-] (\xc,\yb) -- (\xc, \yt);
	}
}
\newcommand{\idsdown}[5]{
\pgfmathsetmacro{\xl}{{#1}}
\pgfmathsetmacro{\yb}{{#2}}
\pgfmathsetmacro{\num}{{#5}}
\pgfmathsetmacro{\yheight}{{#4}}
\pgfmathsetmacro{\yt}{\yb+\yheight};
\pgfmathsetmacro{\xscale}{{#3}};
	\foreach \x in {1,...,\num}{
		\pgfmathparse{\xl+(\x-1)*\xscale}\let\xc\pgfmathresult
		\draw[line width=1pt,-<-] (\xc,\yb) -- (\xc, \yt);
	}
}

\newcommand{\rcross}[4]{
\tikzstyle{cross line}=[preaction={draw=white, -,line width=10pt}];
\pgfmathsetmacro{\xl}{{#1}}
\pgfmathsetmacro{\yb}{{#2}}
\pgfmathsetmacro{\yheight}{{#3}}
\pgfmathsetmacro{\xscale}{{#4}}
\pgfmathsetmacro{\yt}{\yb+\yheight};
\pgfmathsetmacro{\xr}{\xscale+\xl};
\pgfmathsetmacro{\yc}{\yb+\yheight/2};
\pgfmathsetmacro{\xc}{\xscale/2+\xl};

	\pgfmathparse{\xl}		\let\Xa\pgfmathresult
    \pgfmathparse{\yt}		\let\Ya\pgfmathresult
    \coordinate (A) at (\Xa,\Ya);
	\pgfmathparse{\xl+(0.3*\xscale)}		\let\Xaa\pgfmathresult
    \pgfmathparse{\yb+(0.9*\yheight)}		\let\Yaa\pgfmathresult
    \coordinate (A') at (\Xaa,\Yaa);
	\pgfmathparse{\xl}		\let\Xb\pgfmathresult
    \pgfmathparse{\yb}		\let\Yb\pgfmathresult
    \coordinate (B) at (\Xb,\Yb);
	\pgfmathparse{\xl+(0.3*\xscale)}		\let\Xbb\pgfmathresult
    \pgfmathparse{\yb+(0.1*\yheight)}		\let\Ybb\pgfmathresult
    \coordinate (B') at (\Xbb,\Ybb);
	\pgfmathparse{\xr}		\let\Xc\pgfmathresult
    \pgfmathparse{\yb}		\let\Yc\pgfmathresult
    \coordinate (C) at (\Xc,\Yc);
	\pgfmathparse{\xl+(0.7*\xscale)}		\let\Xcc\pgfmathresult
    \pgfmathparse{\yb+(0.1*\yheight)}		\let\Ycc\pgfmathresult
    \coordinate (C') at (\Xcc,\Ycc);
	\pgfmathparse{\xr}		\let\Xd\pgfmathresult
    \pgfmathparse{\yt}		\let\Yd\pgfmathresult
    \coordinate (D) at (\Xd,\Yd);
	\pgfmathparse{\xl+(0.7*\xscale)}		\let\Xdd\pgfmathresult
    \pgfmathparse{\yb+(0.9*\yheight)}		\let\Ydd\pgfmathresult
    \coordinate (D') at (\Xdd,\Ydd);
	\pgfmathparse{\xc}		\let\Xe\pgfmathresult
    \pgfmathparse{\yc}		\let\Ye\pgfmathresult
    \coordinate (E) at (\Xe,\Ye);

\draw[line width=1pt] (A) .. controls (B') and (D') .. (C);
\draw (\xc, \yc) node[shape=circle, fill=white] {};
\draw[line width=1pt] (B) .. controls (A') and (C') .. (D);
}

\newcommand{\rcrossup}[4]{
\tikzstyle{cross line}=[preaction={draw=white, -,line width=10pt}];
\pgfmathsetmacro{\xl}{{#1}}
\pgfmathsetmacro{\yb}{{#2}}
\pgfmathsetmacro{\yheight}{{#3}}
\pgfmathsetmacro{\xscale}{{#4}}
\pgfmathsetmacro{\yt}{\yb+\yheight};
\pgfmathsetmacro{\xr}{\xscale+\xl};
\pgfmathsetmacro{\yc}{\yb+\yheight/2};
\pgfmathsetmacro{\xc}{\xscale/2+\xl};

	\pgfmathparse{\xl}		\let\Xa\pgfmathresult
    \pgfmathparse{\yt}		\let\Ya\pgfmathresult
    \coordinate (A) at (\Xa,\Ya);
	\pgfmathparse{\xl+(0.3*\xscale)}		\let\Xaa\pgfmathresult
    \pgfmathparse{\yb+(0.9*\yheight)}		\let\Yaa\pgfmathresult
    \coordinate (A') at (\Xaa,\Yaa);
	\pgfmathparse{\xl}		\let\Xb\pgfmathresult
    \pgfmathparse{\yb}		\let\Yb\pgfmathresult
    \coordinate (B) at (\Xb,\Yb);
	\pgfmathparse{\xl+(0.3*\xscale)}		\let\Xbb\pgfmathresult
    \pgfmathparse{\yb+(0.1*\yheight)}		\let\Ybb\pgfmathresult
    \coordinate (B') at (\Xbb,\Ybb);
	\pgfmathparse{\xr}		\let\Xc\pgfmathresult
    \pgfmathparse{\yb}		\let\Yc\pgfmathresult
    \coordinate (C) at (\Xc,\Yc);
	\pgfmathparse{\xl+(0.7*\xscale)}		\let\Xcc\pgfmathresult
    \pgfmathparse{\yb+(0.1*\yheight)}		\let\Ycc\pgfmathresult
    \coordinate (C') at (\Xcc,\Ycc);
	\pgfmathparse{\xr}		\let\Xd\pgfmathresult
    \pgfmathparse{\yt}		\let\Yd\pgfmathresult
    \coordinate (D) at (\Xd,\Yd);
	\pgfmathparse{\xl+(0.7*\xscale)}		\let\Xdd\pgfmathresult
    \pgfmathparse{\yb+(0.9*\yheight)}		\let\Ydd\pgfmathresult
    \coordinate (D') at (\Xdd,\Ydd);
	\pgfmathparse{\xc}		\let\Xe\pgfmathresult
    \pgfmathparse{\yc}		\let\Ye\pgfmathresult
    \coordinate (E) at (\Xe,\Ye);

\draw[line width=1pt, <-] (A) .. controls (B') and (D') .. (C);
\draw (\xc, \yc) node[shape=circle, fill=white] {};
\draw[line width=1pt, ->] (B) .. controls (A') and (C') .. (D);
}
\newcommand{\NESE}[4]{
\tikzstyle{cross line}=[preaction={draw=white, -,line width=10pt}];
\pgfmathsetmacro{\xl}{{#1}}
\pgfmathsetmacro{\yb}{{#2}}
\pgfmathsetmacro{\yheight}{{#3}}
\pgfmathsetmacro{\xscale}{{#4}}
\pgfmathsetmacro{\yt}{\yb+\yheight};
\pgfmathsetmacro{\xr}{\xscale+\xl};
\pgfmathsetmacro{\yc}{\yb+\yheight/2};
\pgfmathsetmacro{\xc}{\xscale/2+\xl};

	\pgfmathparse{\xl}		\let\Xa\pgfmathresult
    \pgfmathparse{\yt}		\let\Ya\pgfmathresult
    \coordinate (A) at (\Xa,\Ya);
	\pgfmathparse{\xl+(0.3*\xscale)}		\let\Xaa\pgfmathresult
    \pgfmathparse{\yb+(0.9*\yheight)}		\let\Yaa\pgfmathresult
    \coordinate (A') at (\Xaa,\Yaa);
	\pgfmathparse{\xl}		\let\Xb\pgfmathresult
    \pgfmathparse{\yb}		\let\Yb\pgfmathresult
    \coordinate (B) at (\Xb,\Yb);
	\pgfmathparse{\xl+(0.3*\xscale)}		\let\Xbb\pgfmathresult
    \pgfmathparse{\yb+(0.1*\yheight)}		\let\Ybb\pgfmathresult
    \coordinate (B') at (\Xbb,\Ybb);
	\pgfmathparse{\xr}		\let\Xc\pgfmathresult
    \pgfmathparse{\yb}		\let\Yc\pgfmathresult
    \coordinate (C) at (\Xc,\Yc);
	\pgfmathparse{\xl+(0.7*\xscale)}		\let\Xcc\pgfmathresult
    \pgfmathparse{\yb+(0.1*\yheight)}		\let\Ycc\pgfmathresult
    \coordinate (C') at (\Xcc,\Ycc);
	\pgfmathparse{\xr}		\let\Xd\pgfmathresult
    \pgfmathparse{\yt}		\let\Yd\pgfmathresult
    \coordinate (D) at (\Xd,\Yd);
	\pgfmathparse{\xl+(0.7*\xscale)}		\let\Xdd\pgfmathresult
    \pgfmathparse{\yb+(0.9*\yheight)}		\let\Ydd\pgfmathresult
    \coordinate (D') at (\Xdd,\Ydd);
	\pgfmathparse{\xc}		\let\Xe\pgfmathresult
    \pgfmathparse{\yc}		\let\Ye\pgfmathresult
    \coordinate (E) at (\Xe,\Ye);

\draw[line width=1pt, ->] (A) .. controls (B') and (D') .. (C);
\draw (\xc, \yc) node[shape=circle, fill=white] {};
\draw[line width=1pt, ->] (B) .. controls (A') and (C') .. (D);
}
\newcommand{\SWNW}[4]{
\tikzstyle{cross line}=[preaction={draw=white, -,line width=10pt}];
\pgfmathsetmacro{\xl}{{#1}}
\pgfmathsetmacro{\yb}{{#2}}
\pgfmathsetmacro{\yheight}{{#3}}
\pgfmathsetmacro{\xscale}{{#4}}
\pgfmathsetmacro{\yt}{\yb+\yheight};
\pgfmathsetmacro{\xr}{\xscale+\xl};
\pgfmathsetmacro{\yc}{\yb+\yheight/2};
\pgfmathsetmacro{\xc}{\xscale/2+\xl};

	\pgfmathparse{\xl}		\let\Xa\pgfmathresult
    \pgfmathparse{\yt}		\let\Ya\pgfmathresult
    \coordinate (A) at (\Xa,\Ya);
	\pgfmathparse{\xl+(0.3*\xscale)}		\let\Xaa\pgfmathresult
    \pgfmathparse{\yb+(0.9*\yheight)}		\let\Yaa\pgfmathresult
    \coordinate (A') at (\Xaa,\Yaa);
	\pgfmathparse{\xl}		\let\Xb\pgfmathresult
    \pgfmathparse{\yb}		\let\Yb\pgfmathresult
    \coordinate (B) at (\Xb,\Yb);
	\pgfmathparse{\xl+(0.3*\xscale)}		\let\Xbb\pgfmathresult
    \pgfmathparse{\yb+(0.1*\yheight)}		\let\Ybb\pgfmathresult
    \coordinate (B') at (\Xbb,\Ybb);
	\pgfmathparse{\xr}		\let\Xc\pgfmathresult
    \pgfmathparse{\yb}		\let\Yc\pgfmathresult
    \coordinate (C) at (\Xc,\Yc);
	\pgfmathparse{\xl+(0.7*\xscale)}		\let\Xcc\pgfmathresult
    \pgfmathparse{\yb+(0.1*\yheight)}		\let\Ycc\pgfmathresult
    \coordinate (C') at (\Xcc,\Ycc);
	\pgfmathparse{\xr}		\let\Xd\pgfmathresult
    \pgfmathparse{\yt}		\let\Yd\pgfmathresult
    \coordinate (D) at (\Xd,\Yd);
	\pgfmathparse{\xl+(0.7*\xscale)}		\let\Xdd\pgfmathresult
    \pgfmathparse{\yb+(0.9*\yheight)}		\let\Ydd\pgfmathresult
    \coordinate (D') at (\Xdd,\Ydd);
	\pgfmathparse{\xc}		\let\Xe\pgfmathresult
    \pgfmathparse{\yc}		\let\Ye\pgfmathresult
    \coordinate (E) at (\Xe,\Ye);

\draw[line width=1pt, <-] (A) .. controls (B') and (D') .. (C);
\draw (\xc, \yc) node[shape=circle, fill=white] {};
\draw[line width=1pt, <-] (B) .. controls (A') and (C') .. (D);
}

\newcommand{\SENE}[4]{
\tikzstyle{cross line}=[preaction={draw=white, -,line width=10pt}];
\pgfmathsetmacro{\xl}{{#1}}
\pgfmathsetmacro{\yb}{{#2}}
\pgfmathsetmacro{\yheight}{{#3}}
\pgfmathsetmacro{\xscale}{{#4}}
\pgfmathsetmacro{\yt}{\yb+\yheight};
\pgfmathsetmacro{\xr}{\xscale+\xl};
\pgfmathsetmacro{\yc}{\yb+\yheight/2};
\pgfmathsetmacro{\xc}{\xscale/2+\xl};

	\pgfmathparse{\xl}		\let\Xa\pgfmathresult
    \pgfmathparse{\yt}		\let\Ya\pgfmathresult
    \coordinate (A) at (\Xa,\Ya);
	\pgfmathparse{\xl+(0.3*\xscale)}		\let\Xaa\pgfmathresult
    \pgfmathparse{\yb+(0.9*\yheight)}		\let\Yaa\pgfmathresult
    \coordinate (A') at (\Xaa,\Yaa);
	\pgfmathparse{\xl}		\let\Xb\pgfmathresult
    \pgfmathparse{\yb}		\let\Yb\pgfmathresult
    \coordinate (B) at (\Xb,\Yb);
	\pgfmathparse{\xl+(0.3*\xscale)}		\let\Xbb\pgfmathresult
    \pgfmathparse{\yb+(0.1*\yheight)}		\let\Ybb\pgfmathresult
    \coordinate (B') at (\Xbb,\Ybb);
	\pgfmathparse{\xr}		\let\Xc\pgfmathresult
    \pgfmathparse{\yb}		\let\Yc\pgfmathresult
    \coordinate (C) at (\Xc,\Yc);
	\pgfmathparse{\xl+(0.7*\xscale)}		\let\Xcc\pgfmathresult
    \pgfmathparse{\yb+(0.1*\yheight)}		\let\Ycc\pgfmathresult
    \coordinate (C') at (\Xcc,\Ycc);
	\pgfmathparse{\xr}		\let\Xd\pgfmathresult
    \pgfmathparse{\yt}		\let\Yd\pgfmathresult
    \coordinate (D) at (\Xd,\Yd);
	\pgfmathparse{\xl+(0.7*\xscale)}		\let\Xdd\pgfmathresult
    \pgfmathparse{\yb+(0.9*\yheight)}		\let\Ydd\pgfmathresult
    \coordinate (D') at (\Xdd,\Ydd);
	\pgfmathparse{\xc}		\let\Xe\pgfmathresult
    \pgfmathparse{\yc}		\let\Ye\pgfmathresult
    \coordinate (E) at (\Xe,\Ye);

\draw[line width=1pt, ->] (B) .. controls (A') and (C') .. (D);
\draw (\xc, \yc) node[shape=circle, fill=white] {};
\draw[line width=1pt, ->] (A) .. controls (B') and (D') .. (C);
}
\newcommand{\NWSW}[4]{
\tikzstyle{cross line}=[preaction={draw=white, -,line width=10pt}];
\pgfmathsetmacro{\xl}{{#1}}
\pgfmathsetmacro{\yb}{{#2}}
\pgfmathsetmacro{\yheight}{{#3}}
\pgfmathsetmacro{\xscale}{{#4}}
\pgfmathsetmacro{\yt}{\yb+\yheight};
\pgfmathsetmacro{\xr}{\xscale+\xl};
\pgfmathsetmacro{\yc}{\yb+\yheight/2};
\pgfmathsetmacro{\xc}{\xscale/2+\xl};

	\pgfmathparse{\xl}		\let\Xa\pgfmathresult
    \pgfmathparse{\yt}		\let\Ya\pgfmathresult
    \coordinate (A) at (\Xa,\Ya);
	\pgfmathparse{\xl+(0.3*\xscale)}		\let\Xaa\pgfmathresult
    \pgfmathparse{\yb+(0.9*\yheight)}		\let\Yaa\pgfmathresult
    \coordinate (A') at (\Xaa,\Yaa);
	\pgfmathparse{\xl}		\let\Xb\pgfmathresult
    \pgfmathparse{\yb}		\let\Yb\pgfmathresult
    \coordinate (B) at (\Xb,\Yb);
	\pgfmathparse{\xl+(0.3*\xscale)}		\let\Xbb\pgfmathresult
    \pgfmathparse{\yb+(0.1*\yheight)}		\let\Ybb\pgfmathresult
    \coordinate (B') at (\Xbb,\Ybb);
	\pgfmathparse{\xr}		\let\Xc\pgfmathresult
    \pgfmathparse{\yb}		\let\Yc\pgfmathresult
    \coordinate (C) at (\Xc,\Yc);
	\pgfmathparse{\xl+(0.7*\xscale)}		\let\Xcc\pgfmathresult
    \pgfmathparse{\yb+(0.1*\yheight)}		\let\Ycc\pgfmathresult
    \coordinate (C') at (\Xcc,\Ycc);
	\pgfmathparse{\xr}		\let\Xd\pgfmathresult
    \pgfmathparse{\yt}		\let\Yd\pgfmathresult
    \coordinate (D) at (\Xd,\Yd);
	\pgfmathparse{\xl+(0.7*\xscale)}		\let\Xdd\pgfmathresult
    \pgfmathparse{\yb+(0.9*\yheight)}		\let\Ydd\pgfmathresult
    \coordinate (D') at (\Xdd,\Ydd);
	\pgfmathparse{\xc}		\let\Xe\pgfmathresult
    \pgfmathparse{\yc}		\let\Ye\pgfmathresult
    \coordinate (E) at (\Xe,\Ye);

\draw[line width=1pt, <-] (B) .. controls (A') and (C') .. (D);
\draw (\xc, \yc) node[shape=circle, fill=white] {};
\draw[line width=1pt, <-] (A) .. controls (B') and (D') .. (C);
}

\newcommand{\rcrossdown}[4]{
\tikzstyle{cross line}=[preaction={draw=white, -,line width=10pt}];
\pgfmathsetmacro{\xl}{{#1}}
\pgfmathsetmacro{\yb}{{#2}}
\pgfmathsetmacro{\yheight}{{#3}}
\pgfmathsetmacro{\xscale}{{#4}}
\pgfmathsetmacro{\yt}{\yb+\yheight};
\pgfmathsetmacro{\xr}{\xscale+\xl};
\pgfmathsetmacro{\yc}{\yb+\yheight/2};
\pgfmathsetmacro{\xc}{\xscale/2+\xl};

	\pgfmathparse{\xl}		\let\Xa\pgfmathresult
    \pgfmathparse{\yt}		\let\Ya\pgfmathresult
    \coordinate (A) at (\Xa,\Ya);
	\pgfmathparse{\xl+(0.3*\xscale)}		\let\Xaa\pgfmathresult
    \pgfmathparse{\yb+(0.9*\yheight)}		\let\Yaa\pgfmathresult
    \coordinate (A') at (\Xaa,\Yaa);
	\pgfmathparse{\xl}		\let\Xb\pgfmathresult
    \pgfmathparse{\yb}		\let\Yb\pgfmathresult
    \coordinate (B) at (\Xb,\Yb);
	\pgfmathparse{\xl+(0.3*\xscale)}		\let\Xbb\pgfmathresult
    \pgfmathparse{\yb+(0.1*\yheight)}		\let\Ybb\pgfmathresult
    \coordinate (B') at (\Xbb,\Ybb);
	\pgfmathparse{\xr}		\let\Xc\pgfmathresult
    \pgfmathparse{\yb}		\let\Yc\pgfmathresult
    \coordinate (C) at (\Xc,\Yc);
	\pgfmathparse{\xl+(0.7*\xscale)}		\let\Xcc\pgfmathresult
    \pgfmathparse{\yb+(0.1*\yheight)}		\let\Ycc\pgfmathresult
    \coordinate (C') at (\Xcc,\Ycc);
	\pgfmathparse{\xr}		\let\Xd\pgfmathresult
    \pgfmathparse{\yt}		\let\Yd\pgfmathresult
    \coordinate (D) at (\Xd,\Yd);
	\pgfmathparse{\xl+(0.7*\xscale)}		\let\Xdd\pgfmathresult
    \pgfmathparse{\yb+(0.9*\yheight)}		\let\Ydd\pgfmathresult
    \coordinate (D') at (\Xdd,\Ydd);
	\pgfmathparse{\xc}		\let\Xe\pgfmathresult
    \pgfmathparse{\yc}		\let\Ye\pgfmathresult
    \coordinate (E) at (\Xe,\Ye);

\draw[line width=1pt, ->] (A) .. controls (B') and (D') .. (C);
\draw (\xc, \yc) node[shape=circle, fill=white] {};
\draw[line width=1pt, <-] (B) .. controls (A') and (C') .. (D);
}

\newcommand{\lcross}[4]{
\tikzstyle{cross line}=[preaction={draw=white, -,line width=10pt}];
\pgfmathsetmacro{\xl}{{#1}}
\pgfmathsetmacro{\yb}{{#2}}
\pgfmathsetmacro{\yheight}{{#3}}
\pgfmathsetmacro{\xscale}{{#4}}
\pgfmathsetmacro{\yt}{\yb+\yheight};
\pgfmathsetmacro{\xr}{\xscale+\xl};
\pgfmathsetmacro{\yc}{\yb+\yheight/2};
\pgfmathsetmacro{\xc}{\xscale/2+\xl};

	\pgfmathparse{\xl}		\let\Xa\pgfmathresult
    \pgfmathparse{\yt}		\let\Ya\pgfmathresult
    \coordinate (A) at (\Xa,\Ya);
	\pgfmathparse{\xl+(0.3*\xscale)}		\let\Xaa\pgfmathresult
    \pgfmathparse{\yb+(0.9*\yheight)}		\let\Yaa\pgfmathresult
    \coordinate (A') at (\Xaa,\Yaa);
	\pgfmathparse{\xl}		\let\Xb\pgfmathresult
    \pgfmathparse{\yb}		\let\Yb\pgfmathresult
    \coordinate (B) at (\Xb,\Yb);
	\pgfmathparse{\xl+(0.3*\xscale)}		\let\Xbb\pgfmathresult
    \pgfmathparse{\yb+(0.1*\yheight)}		\let\Ybb\pgfmathresult
    \coordinate (B') at (\Xbb,\Ybb);
	\pgfmathparse{\xr}		\let\Xc\pgfmathresult
    \pgfmathparse{\yb}		\let\Yc\pgfmathresult
    \coordinate (C) at (\Xc,\Yc);
	\pgfmathparse{\xl+(0.7*\xscale)}		\let\Xcc\pgfmathresult
    \pgfmathparse{\yb+(0.1*\yheight)}		\let\Ycc\pgfmathresult
    \coordinate (C') at (\Xcc,\Ycc);
	\pgfmathparse{\xr}		\let\Xd\pgfmathresult
    \pgfmathparse{\yt}		\let\Yd\pgfmathresult
    \coordinate (D) at (\Xd,\Yd);
	\pgfmathparse{\xl+(0.7*\xscale)}		\let\Xdd\pgfmathresult
    \pgfmathparse{\yb+(0.9*\yheight)}		\let\Ydd\pgfmathresult
    \coordinate (D') at (\Xdd,\Ydd);
	\pgfmathparse{\xc}		\let\Xe\pgfmathresult
    \pgfmathparse{\yc}		\let\Ye\pgfmathresult
    \coordinate (E) at (\Xe,\Ye);

\draw[line width=1pt] (B) .. controls (A') and (C') .. (D);
\draw (\xc, \yc) node[shape=circle, fill=white] {};
\draw[line width=1pt] (A) .. controls (B') and (D') .. (C);
}

\newcommand{\lcrossup}[4]{
\tikzstyle{cross line}=[preaction={draw=white, -,line width=10pt}];
\pgfmathsetmacro{\xl}{{#1}}
\pgfmathsetmacro{\yb}{{#2}}
\pgfmathsetmacro{\yheight}{{#3}}
\pgfmathsetmacro{\xscale}{{#4}}
\pgfmathsetmacro{\yt}{\yb+\yheight};
\pgfmathsetmacro{\xr}{\xscale+\xl};
\pgfmathsetmacro{\yc}{\yb+\yheight/2};
\pgfmathsetmacro{\xc}{\xscale/2+\xl};

	\pgfmathparse{\xl}		\let\Xa\pgfmathresult
    \pgfmathparse{\yt}		\let\Ya\pgfmathresult
    \coordinate (A) at (\Xa,\Ya);
	\pgfmathparse{\xl+(0.3*\xscale)}		\let\Xaa\pgfmathresult
    \pgfmathparse{\yb+(0.9*\yheight)}		\let\Yaa\pgfmathresult
    \coordinate (A') at (\Xaa,\Yaa);
	\pgfmathparse{\xl}		\let\Xb\pgfmathresult
    \pgfmathparse{\yb}		\let\Yb\pgfmathresult
    \coordinate (B) at (\Xb,\Yb);
	\pgfmathparse{\xl+(0.3*\xscale)}		\let\Xbb\pgfmathresult
    \pgfmathparse{\yb+(0.1*\yheight)}		\let\Ybb\pgfmathresult
    \coordinate (B') at (\Xbb,\Ybb);
	\pgfmathparse{\xr}		\let\Xc\pgfmathresult
    \pgfmathparse{\yb}		\let\Yc\pgfmathresult
    \coordinate (C) at (\Xc,\Yc);
	\pgfmathparse{\xl+(0.7*\xscale)}		\let\Xcc\pgfmathresult
    \pgfmathparse{\yb+(0.1*\yheight)}		\let\Ycc\pgfmathresult
    \coordinate (C') at (\Xcc,\Ycc);
	\pgfmathparse{\xr}		\let\Xd\pgfmathresult
    \pgfmathparse{\yt}		\let\Yd\pgfmathresult
    \coordinate (D) at (\Xd,\Yd);
	\pgfmathparse{\xl+(0.7*\xscale)}		\let\Xdd\pgfmathresult
    \pgfmathparse{\yb+(0.9*\yheight)}		\let\Ydd\pgfmathresult
    \coordinate (D') at (\Xdd,\Ydd);
	\pgfmathparse{\xc}		\let\Xe\pgfmathresult
    \pgfmathparse{\yc}		\let\Ye\pgfmathresult
    \coordinate (E) at (\Xe,\Ye);

\draw[line width=1pt, ->] (B) .. controls (A') and (C') .. (D);
\draw (\xc, \yc) node[shape=circle, fill=white] {};
\draw[line width=1pt, <-] (A) .. controls (B') and (D') .. (C);
}

\newcommand{\lcrossdown}[4]{
\tikzstyle{cross line}=[preaction={draw=white, -,line width=10pt}];
\pgfmathsetmacro{\xl}{{#1}}
\pgfmathsetmacro{\yb}{{#2}}
\pgfmathsetmacro{\yheight}{{#3}}
\pgfmathsetmacro{\xscale}{{#4}}
\pgfmathsetmacro{\yt}{\yb+\yheight};
\pgfmathsetmacro{\xr}{\xscale+\xl};
\pgfmathsetmacro{\yc}{\yb+\yheight/2};
\pgfmathsetmacro{\xc}{\xscale/2+\xl};

	\pgfmathparse{\xl}		\let\Xa\pgfmathresult
    \pgfmathparse{\yt}		\let\Ya\pgfmathresult
    \coordinate (A) at (\Xa,\Ya);
	\pgfmathparse{\xl+(0.3*\xscale)}		\let\Xaa\pgfmathresult
    \pgfmathparse{\yb+(0.9*\yheight)}		\let\Yaa\pgfmathresult
    \coordinate (A') at (\Xaa,\Yaa);
	\pgfmathparse{\xl}		\let\Xb\pgfmathresult
    \pgfmathparse{\yb}		\let\Yb\pgfmathresult
    \coordinate (B) at (\Xb,\Yb);
	\pgfmathparse{\xl+(0.3*\xscale)}		\let\Xbb\pgfmathresult
    \pgfmathparse{\yb+(0.1*\yheight)}		\let\Ybb\pgfmathresult
    \coordinate (B') at (\Xbb,\Ybb);
	\pgfmathparse{\xr}		\let\Xc\pgfmathresult
    \pgfmathparse{\yb}		\let\Yc\pgfmathresult
    \coordinate (C) at (\Xc,\Yc);
	\pgfmathparse{\xl+(0.7*\xscale)}		\let\Xcc\pgfmathresult
    \pgfmathparse{\yb+(0.1*\yheight)}		\let\Ycc\pgfmathresult
    \coordinate (C') at (\Xcc,\Ycc);
	\pgfmathparse{\xr}		\let\Xd\pgfmathresult
    \pgfmathparse{\yt}		\let\Yd\pgfmathresult
    \coordinate (D) at (\Xd,\Yd);
	\pgfmathparse{\xl+(0.7*\xscale)}		\let\Xdd\pgfmathresult
    \pgfmathparse{\yb+(0.9*\yheight)}		\let\Ydd\pgfmathresult
    \coordinate (D') at (\Xdd,\Ydd);
	\pgfmathparse{\xc}		\let\Xe\pgfmathresult
    \pgfmathparse{\yc}		\let\Ye\pgfmathresult
    \coordinate (E) at (\Xe,\Ye);

\draw[line width=1pt, <-] (B) .. controls (A') and (C') .. (D);
\draw (\xc, \yc) node[shape=circle, fill=white] {};
\draw[line width=1pt, ->] (A) .. controls (B') and (D') .. (C);
}

\newcommand{\standin}[5]{
\tikzstyle{cross line}=[preaction={draw=white, -,line width=10pt}];
\pgfmathsetmacro{\xl}{{#1}}
\pgfmathsetmacro{\yb}{{#2}}
\pgfmathsetmacro{\yheight}{{#3}}
\pgfmathsetmacro{\xscale}{{#4}}
\pgfmathsetmacro{\yt}{\yb+\yheight};
\pgfmathsetmacro{\xr}{\xscale+\xl};
\pgfmathsetmacro{\yc}{\yb+\yheight/2};
\pgfmathsetmacro{\xc}{\xscale/2+\xl};
\draw[line width=1pt] (\xl,\yb) rectangle (\xr, \yt);
\draw (\xc, \yc) node {\textbf{#5}};
}

\title[]{The knot invariant associated to two-parameter quantum algebras}
\makeindex

\author[Z. Fan and J. Xing]{Zhaobing Fan and Junjing Xing}
\address{College of Intelligent Systems Science and Engineering, Harbin Engineering University, Harbin,  China, 150001}
\address{College of Mathematical Sciences, Harbin Engineering University, Harbin,  China, 150001}
\email{fanzhaobing@hrbeu.edu.cn  (Fan)}

\address{College of Intelligent Systems Science and Engineering, Harbin Engineering University, Harbin,  China, 150001}
\email{xingjunjing2017@hrbeu.edu.cn (Xing)}

\date{\today}

\begin{document}
\begin{abstract}
Using the skew-Hopf pairing, we obtain $\mathcal{R}$-matrix
for the two-parameter quantum algebra $U_{v,t}$.
We further construct a strict monoidal functor $\mathcal{T}$ from
the tangle category $(\rm{OTa},\otimes, \emptyset)$ to the
category $(\rm{Mod},\otimes, \mathbb{Q}(v,t))$ of $U_{v,t}$-modules .
As a consequence, the quantum knot invariant of the tangle $L$ of type $(n,n)$
is obtained by the action of $\mathcal{T}$ on the closure $\tilde{L}$ of $L$.
\end{abstract}
\maketitle

\setcounter{tocdepth}{1}
\tableofcontents

\section{Introduction}
Knot theory is the mathematical discipline with the unusually diverse applications,
such as statistical mechanics \cite{Kau11},
symbolic logic and set theory \cite{Kau}, quantum field theory \cite{W}, quantum computing \cite{N}, etc.
Reshetikhin and Turaev \cite{RT90} related quantum algebras to knot invariants,
often referred to as quantum invariants.
They generalized the Jones polynomial of links and the related Jones-Conway and Kauffman polynomials to the case of graphs in $\mathbb{R}^3$.
Inspired by the result, a large number of researchers began to pay attention to
quantum invariants, such as Zhang \cite{ZGB}, Kauffman \cite{KM} and Clark \cite{C}.

In \cite{FL}, the first author and Li provided a noval presentation of
the two-parameter quantum algebra $U_{v,t}(\mathfrak g)$ by a geometric approach, where
both parameters $v$ and $t$ have geometric meaning.
Moreover, this presentation unifies the various quantum algebras in literature.
By various specialization, one can obtain one-parameter quantum algebras \cite{Lusztigbook},
two-parameter quantum algebras \cite{BW}, quantum superalgebras \cite{CFYW} and multi-parameter quantum algebras \cite{HPR}.
It is a natural question that whether this two-parameter quantum algebra
provide a new knot invariant.
This is what we will explore in this paper.
In \cite{C}, Clark gave a partial answer to this question by proving that
the quantum knot invariants for $\mathfrak{osp}(1|2n)$ and $\mathfrak{so}(2n+1)$ are essentially the same.

The solution of the Yang-Baxter equation is called $\mathcal{R}$-matrix
which connects quantum algebras to the knot invariant theory.
In the construction of quantum knot invariants,
invariance under the Reidemeister III move holds naturally by
attaching a copy of the $\mathcal{R}$-matrix to each crossing.
The construction of $\mathcal{R}$-matrices for various quantum groups is very meaningful and interesting.
In \cite[Chapter 32]{Lusztigbook}, Lusztig provided a framework to construct the $\mathcal{R}$-matrix of one-parameter quantum algebras
via the quasi-$\mathcal{R}$-matrix.
%and a function $f:X \times X \rightarrow \mathbb{Q}$ satisfying some identities.
In \cite{FX}, we defined a skew-Hopf pairing on the deformed quantum algebra
which unifies various quantum algebras in literatures,
such as one-parameter quantum algebras \cite{Lusztigbook},
two-parameter quantum algebras \cite{BW}, quantum superalgebras \cite{CFYW} and multi-parameter quantum algebras \cite{HPR}.
This provided a tool to construct $\mathcal{R}$-matrix for quantum algebras other than one-parameter one.
For instance, in \cite{BW}, Benkart and Witherspoon constructed the $\mathcal{R}$-matrix
for two-parameter quantum algebra $U_{r,s}(\mathfrak{sl}_n)$
by the Hopf pairing  which can be viewed as a special case of that in \cite{FX}.

In this paper, we construct the $\mathcal{R}$-matrix for the two-parameter quantum algebra $U_{v,t}$ by using skew-Hopf pairing.
This recovers the constructions of $\mathcal{R}$-matrix in \cite{Lusztigbook} and \cite{BW} under certain assumptions.
We further provide the functor $\mathcal{T}:(\rm{OTa}, \otimes, \emptyset)\rightarrow (\rm{Mod},\otimes, \mathbb{Q}(v,t))$,
where $\rm{OTa}$ and $\rm{Mod}$ are the categories of tangles and $U_{v,t}$-modules, respectively.
This produces the machinery and correspondence for the construction of quantum invariants
via representations of two-parameter quantum algebra $U_{v,t}$.
Furthermore, given a tangle $L$ of type $(n,n)$, we can get an endomorphism $\mathcal{T}(\tilde{L})$
of the ground field $\mathbb{Q}(v,t)$, where $\tilde{L}$ is the closure of $L$.

This paper is organized as follows.
In Section \ref{sec:twoparameter}, we recall the definition of two-parameter quantum algebra $U_{v,t}$ from \cite{FL}
and formulate the quasi-$\mathcal{R}$-matrix $\Theta$ of $U_{v,t}$. Part of the results are new for two-parameter quantum groups.
In Section \ref{sec:rmatrix}, we construct the $\mathcal{R}$-matrix of two-parameter quantum algebras $U_{v,t}$. 
In Section \ref{sec:knotinvariant}, we construct the functor between the categories of tangles and the categories of $U_{v,t}$-modules.

\noindent\textbf{Acknowledgements.}
Z. Fan is partially supported by the NSF of China grant 11671108,
the NSF of Heilongjiang Province grant JQ2020A001 and
the Fundamental Research Funds for the central universities.

\section{The two-parameter quantum algebra $U_{v,t}$}\label{sec:twoparameter}

We briefly review the definition of the two-parameter quantum algebra $U_{v,t}$ in \cite{FL}.

Given a Cartan datum $(I, \cdot)$, let $\Omega=(\Omega_{ij})_{i,j\in I}$ be an integer matrix satisfying that
\begin{itemize}
  \item[(a)] $\Omega_{ii} \in \mathbb{Z}_{>0}$, $\Omega_{ij}\in \mathbb{Z}_{\leq 0}$ for all $i\neq j \in I$;
  \item[(b)] $\frac{\Omega_{ij}+\Omega_{ji}}{\Omega_{ii}}\in \mathbb{Z}_{\leq 0}$ for all $i\neq j \in I$;
  \item[(c)] the greatest common divisor of all $\Omega_{ii}$ is equal to 1.
\end{itemize}

To $\Omega$, we associate the following three bilinear forms on $\mathbb{Z}[I]$.
\begin{eqnarray*}
\langle i, j\rangle &=&
\Omega_{ij}, \quad \hspace{45pt}\forall i, j\in I. \label{eq47}\\
\begin{bmatrix} i,j \end{bmatrix}&=& 2\delta_{ij} \Omega_{ii} -\Omega_{ij}, \quad \forall i, j\in I. \label{eq48}\\
i\cdot j&=&\langle i, j\rangle +\langle j,i\rangle,  \quad \forall i, j\in I. \label{eq49}
%(i,j)&=&\langle j,i\rangle-\langle i, j\rangle,  \quad \forall i, j\in I.
\end{eqnarray*}

\subsection{The free algebra $'\mathfrak{f}$}\label{sec:freealgebra}
For indeterminates $v$ and $t$,
 we set $v_i=v^{i\cdot i/2}$ and $t_i=t^{i\cdot i/2}$.
 Denoted by $v_{\nu}=\prod_i v_i^{\nu_i},\ t_{\nu}=\prod_i t_i^{\nu_i}$
 and ${\rm tr}(\nu) =\sum_{i\in I} \nu_i \in \mathbb{N}$,
 for any $\nu=\sum \nu_i i\in \mathbb{N}[I]$.

Let ${}'\! \mathfrak{f}$ be the free unital associative algebra over $\mathbb{Q}(v,t)$
generated by the symbols $ \theta_i,\ \forall i\in I$.
By setting the degree of the generator $\theta_i$ to be $i$,
the algebra ${}'\! \mathfrak{f}$ becomes an $\mathbb{N}[I]$-graded algebra.
For any $\nu\in \mathbb{N}[I]$, we denote by ${}'\!\mathfrak{f}_{\nu}$
the subspace of all homogenous elements of degree $\nu$.
We have ${}'\! \mathfrak{f}=\oplus_{\nu\in \mathbb{N}[I]}{}'\! \mathfrak{f}_{\nu}$
and denote by $|x|$ the degree of a homogenous element $x\in {}'\! \mathfrak{f}$.

\subsubsection{The tensor product ${}'\!\mathfrak{f}\otimes {}'\!\mathfrak{f}$}
On the tensor product ${}'\!\mathfrak{f}\otimes {}'\!\mathfrak{f}$,
we define an associative $\mathbb{Q}(v,t)$-algebra structure by
\begin{equation*}\label{eq25}
(x_1 \otimes x_2)(y_1 \otimes y_2)
=v^{|y_1|\cdot|x_2|}t^{\langle|y_1|,|x_2|\rangle-\langle |x_2|,|y_1|\rangle}x_1y_1 \otimes x_2y_2,
\end{equation*}
for homogeneous elements $x_1,x_2,y_1$ and $y_2$ in ${}'\!\mathfrak{f}$.
It is associative since the forms $\langle , \rangle$  and $``\cdot "$  are bilinear.

Let $r: {}'\!\mathfrak{f}\rightarrow  {}'\!\mathfrak{f}\otimes {}'\!\mathfrak{f}$
be the $\mathbb{Q}(v,t)$-algebra homomorphism such that
$$r(\theta_i)=\theta_i\otimes 1+1 \otimes \theta_i,\quad {\rm for\ all}\ i\in I.$$

\begin{prop}[]\label{prop:bilinearform}\cite[Proposition 13]{FL}
  There is a unique symmetric bilinear form {\rm (,)} on ${}'\!\mathfrak{f}$ with values in $\mathbb{Q}(v,t)$ such that
  \begin{itemize}
    \item[(a)] $(1,1)=1$;
    \item[(b)] $(\theta_i, \theta_j)=\delta_{ij}\frac{1}{1-v^{-2}_i}$,\quad for all $i, j \in I$;
    \item[(c)] $(x, y'y'')=(r(x), y' \otimes y'')$,\quad for all $x, y', y'' \in {}'\!\mathfrak{f}$;
    \item[(d)] $(x'x'', y)=(x'\otimes x'', r(y))$,\quad for all $x', x'', y \in {}'\!\mathfrak{f}$.
  \end{itemize}
  Here the bilinear form on ${}'\!\mathfrak{f} \otimes {}'\!\mathfrak{f}$ is defined by
  \begin{equation*}\label{eq74}
    (x_1 \otimes x_2, y_1 \otimes y_2)=t^{2[|x_1|,|x_2|]}(x_1, y_1) (x_2, y_2).
  \end{equation*}
\end{prop}

\subsubsection{}
Let $\sigma:{}'\!\mathfrak{f}\rightarrow{}'\!\mathfrak{f}^{op}$ be a twisted  anti-involution such that
\begin{equation}\label{eq:sigma-f}
\sigma(\theta_i)=\theta_i,\quad {\rm and}\quad \sigma(xy)=t^{\langle|y|,|x|\rangle-\langle|x|,|y|\rangle}\sigma(y)\sigma(x)
 \end{equation}
for any homogenous elements $x, y \in {}'\!\mathfrak{f}$.

Let $\rho: {}'\!\mathfrak{f}\otimes{}'\!\mathfrak{f}\rightarrow{}'\!\mathfrak{f}\otimes{}'\!\mathfrak{f}$
be a linear map defined by
$$\rho(x\otimes y)= t^{\langle|y|,|x|\rangle-\langle|x|,|y|\rangle}y\otimes x,\quad \forall x,y \in {}'\!\mathfrak{f}.$$
We set ${}^tr=\rho \circ r$.

\begin{lem}\label{lem:sigma}
We have $r(\sigma(x)) = (\sigma \otimes \sigma)\ {}^t r(x)$, for all $x\in {}'\!\mathfrak{f}$.
\end{lem}

\begin{proof}
We show this lemma by induction on $|x|$. If $x=\theta_i$,
it follows the definition of $\sigma$ and ${}^t r$.

Assume that it holds for any homogenous elements $x'$ and $x''$.
We shall show that it holds for $x=x'x''$.
Let's write $r(x')=\sum x'_1\otimes x'_2$ and $r(x'')=\sum x''_1\otimes x''_2$,
with all factors being homogeneous. Then we have
\begin{equation*}
r(x'x'')=r(x')r(x'')=\sum v^{|x''_1||x'_2|}
t^{\langle|x''_1|,|x'_2|\rangle-\langle|x'_2|,|x''_1|\rangle}
x'_1x''_1\otimes x'_2x''_2.
\end{equation*}
By the definition of ${}^t r$, $(\sigma \otimes \sigma) {}^t r(x'x'')$ is equal to
\begin{equation}\label{eq:sigmatr}
\begin{split}
\sum &  v^{|x''_1||x'_2|}
t^{\langle|x''_1|,|x'_2|\rangle-\langle|x'_2|,|x''_1|\rangle}\\
&t^{\langle|x'_2|+|x''_2|,|x'_1|+|x''_1|\rangle-\langle|x'_1|+|x''_1|,|x'_2|+|x''_2|\rangle}
\sigma(x'_2x''_2)\otimes \sigma(x'_1x''_1)\\
=\sum & v^{|x''_1||x'_2|}
t^{\langle|x''_1|,|x'_2|\rangle-\langle|x'_2|,|x''_1|\rangle
+ \langle|x'_2|+|x''_2|,|x'_1|+|x''_1|\rangle-\langle|x'_1|+|x''_1|,|x'_2|+|x''_2|\rangle}\\
&t^{\langle|x''_2|,|x'_2|\rangle-\langle|x'_2|,|x''_2|\rangle
+\langle|x''_1|,|x'_1|\rangle-\langle|x'_1|,|x''_1|\rangle}
\sigma(x''_2)\sigma(x'_2)\otimes \sigma(x''_1)\sigma(x'_1).
\end{split}
\end{equation}
By hypothesis, we have
$r(\sigma(x'))=\sum t^{ \langle|x'_2|,|x'_1| \rangle -
\langle |x'_1|,|x'_2| \rangle} \sigma(x'_2) \otimes \sigma(x'_1)$
and $r(\sigma(x''))=\sum t^{ \langle |x''_2|,|x''_1| \rangle-
\langle |x''_1|,|x''_2| \rangle} \sigma(x''_2) \otimes \sigma(x''_1)$.
Thus,
\begin{equation}\label{eq:rsigma}
\begin{split}
r(\sigma(x'x''))=&t^{\langle|x''|,|x'|\rangle-\langle|x'|,|x''|\rangle}r(\sigma(x'')\sigma(x'))\\
=&t^{\langle|x''|,|x'|\rangle-\langle|x'|,|x''|\rangle}
\sum t^{\langle|x'_2|,|x'_1|\rangle-\langle|x'_1|,|x'_2|\rangle
+\langle|x''_2|,|x''_1|\rangle-\langle|x''_1|,|x''_2|\rangle}\\
&v^{|x''_1||x'_2|}t^{\langle|x'_2|,|x''_1|\rangle-\langle|x''_1|,|x'_2|\rangle}
\sigma(x''_2) \sigma(x'_2) \otimes \sigma(x''_1) \sigma(x'_1).
\end{split}
\end{equation}
By comparing the exponents of $v$ and $t$ in \eqref{eq:sigmatr}
and \eqref{eq:rsigma} with $|x'|=|x'_1|+|x'_2|$ and $|x''|=|x''_1|+|x''_2|$, we have
\begin{equation*}
r(\sigma(x'x''))= (\sigma \otimes \sigma) {}^t r(x'x'').
\end{equation*}

This finishes the proof.
\end{proof}

\subsubsection{}
Let $\bar{\cdot}: \mathbb{Q}(v,t) \rightarrow \mathbb{Q}(v,t)$
be the unique $\mathbb{Q}$-algebra involution such that
\begin{equation*}\label{eq73}
\overline{v}=v^{-1}\ \rm{and} \ \overline{t}=t.
\end{equation*}

Let $\bar{\cdot}: {}'\! \mathfrak{f}\rightarrow {}'\! \mathfrak{f}$
be the unique $\mathbb{Q}$-algebra involution such that
$$\overline{p \theta_i}=\overline{p} \theta_i, \quad \forall p\in \mathbb{Q}(v,t),\ i \in I.$$
It's clear that $|\overline{x}|=|x|$ for any homogeneous element
$x \in {}'\! \mathfrak{f}$.

Let $\mathfrak{'f} \overline{\otimes} \mathfrak{'f} $ be the
$\mathbb{Q}(v,t)$-vector space  $\mathfrak{'f} \otimes\mathfrak{'f} $
with the associative $\mathbb{Q}(v,t)$-algebra structure given by
$$
(x_1\otimes x_2)(y_1\otimes y_2)=v^{-|y_1||x_2|}t^{\langle|y_1|,|x_2|\rangle-
\langle|x_2|,|y_1|\rangle}x_1y_1\otimes x_2y_2.
$$
Then $\bar{\otimes}:\mathfrak{'f} \otimes\mathfrak{'f} \rightarrow
\mathfrak{'f} \overline{\otimes} \mathfrak{'f} $ is the $\mathbb{Q}(t)$-algebra isomorphism.

Let $\bar{r}:\mathfrak{'f} \rightarrow \mathfrak{'f} \overline{\otimes}\mathfrak{'f} $
be the $\mathbb{Q}(t)$-algebra homorphism defined by
$$\bar{r}(x)=\overline{r(\bar{x})},\quad \forall x\in \mathfrak{'f}.$$
Then we have $\bar{r}(\theta_i)=\theta_i\otimes 1+1\otimes \theta_i$.

\begin{lem} \label{lem:rbar}
For any $x\in \mathfrak{'f}$, by setting $r(x) = \sum x_1\otimes x_2$, we have
$$\bar{r}(x)
=\sum v^{-|x_1|\cdot |x_2|} t^{\langle|x_2|,|x_1|\rangle-\langle|x_1|, |x_2|\rangle} x_2\otimes x_1.$$
\end{lem}
\begin{proof}
By the definition of $\bar{r}$, we shall show that
\begin{equation*}
r(\bar{x})
=\sum v^{|x_1|\cdot |x_2|} t^{\langle|x_2|,|x_1|\rangle-\langle|x_1|, |x_2|\rangle}
\bar{x}_2\otimes \bar{x}_1.
\end{equation*}
The proof for Lemma 5 in \cite{FL} works through if we replace $v$ by $v^{-1}$.
\end{proof}

We note that the coassociativity property of $r$ implies the coassociativity property of $\bar{r}$, i.e.,
$(\bar r\otimes 1)\bar r = (1\otimes \bar r)\bar r$.

\subsubsection{The maps $r_i$ and ${}_ir$}
For any $i\in I$, let $r_i ({\rm resp.}\ {}_ir): {}'\! \mathfrak{f}\rightarrow {}'\! \mathfrak{f}$
be the unique linear map satisfying the following properties.
\begin{equation*}
  \begin{split}
    r_i(1)=0,\ r_i(\theta_j)=\delta_{ij},\ \forall j\in I\ {\rm and}\
    r_i(xy)=v^{i\cdot |y|}t^{\langle |y|,i\rangle-\langle i, |y|\rangle}r_i(x)y+xr_i(y);\\
    {}_ir(1)=0,\ {}_ir(\theta_j)=\delta_{ij},\ \forall j\in I\ {\rm and}\
    {}_ir(xy)={}_ir(x)y+v^{i\cdot |x|}t^{\langle i,|x|\rangle-\langle |x|,i\rangle}x{}_ir(y).
  \end{split}
\end{equation*}
By an induction on $|x|$, we can show that $r(x)=r_i(x)\otimes \theta_i$ (resp. $r(x)= \theta_i\otimes{}_ir(x)$)
plus other terms.

\subsubsection{Quantum serre relations}

Let $\mathfrak{J}$ be the radical of the bilinear form $(-,-)$.
It is clear that $\mathfrak{J}$ is a two-sided ideal of ${}'\!\mathfrak{f}$.
Denote the quotient algebra of ${}'\!\mathfrak{f}$ by
\[\mathfrak{f}={}'\!\mathfrak{f}/\mathfrak{J}.\]

Recall the quantum integers from \cite{FL}.
For any $n\in \mathbb{N}$, we have
$$[n]_{v, t}=\frac{(vt)^n-(vt^{-1})^{-n}}{vt-(vt^{-1})^{-1}},\qquad
[n]^!_{v, t}=\prod_{k=1}^n[k]_{v,t}.$$
Denote by
 $$\theta_i^{(n)}=\frac{\theta_i^n}{[n]^!_{v_i, t_i}}.$$

\begin{prop}\label{prop15}~\cite[Proposition 14]{FL}
  The generators $\theta_i$ of $\mathfrak{f}$ satisfy the following identities.
  $$
  \sum_{p+p'=1-2\frac{i\cdot j}{i\cdot i}}
  (-1)^pt_i^{-p(p'-2\frac{\langle i,j\rangle}{i\cdot i}+
  2\frac{\langle j,i\rangle}{i\cdot i})}\theta_i^{(p)}\theta_j\theta_i^{(p')}=0,
  \quad \forall i\not =j\in I.
  $$
\end{prop}

\subsection{The presentation of the two-parameter quantum algebra $U_{v,t}$}\label{sec:presentation}
By Drinfeld double construction, we get the following presentation of
the entire two-parameter quantum algebra $U_{v,t}$,
generated by symbols $E_i, F_i, K_i^{\pm 1}, K_i'^{\pm 1}, \forall i\in I$,
and subjects to the following relations.
\allowdisplaybreaks
\begin{eqnarray*}
  (R1)& &K_i^{\pm 1}K_i^{\mp 1}=K'^{\pm 1}_iK'^{\mp 1}_i=1.\\
  (R2)& &K_iE_jK^{-1}_i=v^{i\cdot j}t^{\langle j,i\rangle-\langle i,j\rangle} E_j,\ \ K'_iE_jK'^{-1}_i=v^{-i\cdot j}t^{\langle j,i\rangle-\langle i,j\rangle}E_j,\\
      & &K'_iF_jK'^{-1}_i=v^{i\cdot j}t^{\langle i,j\rangle-\langle j,i\rangle} F_j,\ \ K_iF_jK^{-1}_i=v^{-i\cdot j}t^{\langle i,j\rangle-\langle j,i \rangle}F_j.\\
  (R3)& &E_iF_j-F_j E_i=\delta_{ij}\frac{K_i-K'_i}{v_i-v^{-1}_i}.\\
  (R4)& &\sum_{p+p'=1-2\frac{i\cdot j}{i\cdot i}}(-1)^pt_i^{-p(p'-2\frac{\langle i,j\rangle}{i\cdot i}+2\frac{\langle j,i\rangle}{i\cdot i})}
         E_i^{(p)}E_j E_i^{(p')}=0, \quad {\rm if}\ i\not =j,\\
  & &\sum_{p+p'=1-2\frac{i\cdot j}{i\cdot i}}(-1)^pt_i^{-p(p'-2\frac{\langle i,j\rangle}{i\cdot i}+2\frac{\langle j,i\rangle}{i\cdot i})}
         F_i^{(p')}F_j F_i^{(p)}=0, \quad {\rm if}\ i\not =j.\\
\end{eqnarray*}

 The algebra $U_{v,t}$ has a Hopf algebra structure with the comultiplication $\Delta$, the counit $\varepsilon$ and the antipode $S$ given as follows.
  $$\begin{array}{llll}
   &\Delta(K_i^{\pm 1})=K_i^{\pm 1} \otimes K_i^{\pm 1},&\Delta(K'^{\pm 1}_i)=K'^{\pm 1}_i \otimes K'^{\pm 1}_i, &\vspace{4pt}\\
   &\Delta(E_i)=E_i\otimes 1+K_i\otimes E_i,& \Delta(F_i)=1 \otimes F_i+F_i\otimes K'_i, & \vspace{4pt}\\
  &\varepsilon(K_i^{\pm 1})=\varepsilon(K'^{\pm 1}_i)=1,& \varepsilon(E_i)=\varepsilon(F_i)=0,& S(K_i^{\pm 1})=K_i^{\mp 1},\vspace{4pt}\\
  & S(K'^{\pm 1}_i)=K'^{\mp 1}_i,&
  S(E_i)=-K_i^{-1}E_i,&S(F_i)=-F_iK'^{-1}_i.
  \end{array}$$

Let $U_{v,t}^+$(resp. $U_{v,t}^-$) be a subalgebra of $U_{v,t}$ generated by $E_i$(resp. $F_i$).
From the drinfeld double construction, we know that there are two well-defined algebra homomorphisms
$\iota^{+}:\mathfrak{f}\rightarrow U_{v,t}^+\ (x\mapsto x^+)$ and $\iota^{-}:\mathfrak{f}\rightarrow U_{v,t}^-\ (x\mapsto x^-, t\mapsto t^{-1})$
such that $E_i=\theta_i^+, F_i=\theta_i^-$ for all $i\in I$.

Let $\sigma^+: \mathbf{U}^+ \rightarrow \mathbf{U}^+$(resp. $\sigma^-: \mathbf{U}^- \rightarrow \mathbf{U}^-$)
be an anti-involution such that
$\sigma^+(x^+)=\sigma(x)^+$(resp. $\sigma^-(x^-)=\sigma(x)^-$) for all $x\in \mathfrak{f}$.

\begin{lem}\label{lem:S}
For all $x\in \mathfrak{f}$, we have\\
$(\mathrm{i})\quad S(x^+)=(-1)^{{\rm tr} |x|}v^{\frac{|x|\cdot |x|}{2}}v^{-1}_{|x|}K^{-1}_{|x|}\sigma(x)^+$,\\
$(\mathrm{ii})\quad S(x^-)=(-1)^{{\rm tr} |x|}v^{\frac{-|x|\cdot |x|}{2}}v_{|x|}\sigma(x)^-K'^{-1}_{|x|}$.
\end{lem}
\begin{proof}
The proofs of $(\mathrm{i})$ and $(\mathrm{ii})$ are similar.
We shall only show $(\mathrm{i})$ and left $(\mathrm{ii})$ to readers.
If $x=\theta_i$, $(\mathrm{i})$ is straightforward by the definition of $S(E_i)$.

Assume that $(\mathrm{i})$ holds for $x_1$ and $x_2$.
We shall show that it holds for $x=x_1x_2$. Let's write
$S(x^+_1)=(-1)^{{\rm tr} |x_1|}v^{\frac{|x_1|\cdot |x_1|}{2}}v^{-1}_{|x_1|}K^{-1}_{|x_1|}\sigma(x_1)^+$
and $S(x^+_2)=(-1)^{{\rm tr} |x_2|}v^{\frac{|x_2|\cdot |x_2|}{2}}v^{-1}_{|x_2|}K^{-1}_{|x_2|}\sigma(x_2)^+$.
Then, we have
\begin{align*}
&S((x_1x_2)^+)=S(x_2^+)S(x_1^+)\\
&=(-1)^{{\rm tr} (|x_1|+|x_2|)}v^{-1}_{|x_1|+|x_2|}v^{\frac{|x_1|\cdot |x_1|+ |x_2|\cdot |x_2| }{2}}
K^{-1}_{|x_2|}\sigma(x_2)^+K^{-1}_{|x_1|}\sigma(x_1)^+\\
&=(-1)^{{\rm tr} (|x_1|+|x_2|)}v^{-1}_{|x_1|+|x_2|}v^{\frac{(|x_1|+ |x_2|)\cdot (|x_1|+|x_2|)}{2}}
t^{\langle|x_2|,|x_1|\rangle-\langle|x_1|,|x_2|\rangle}K^{-1}_{|x_1|+|x_2|}\sigma(x_2)^+\sigma(x_1)^+\\
&=(-1)^{{\rm tr} (|x_1|+|x_2|)}v^{-1}_{|x_1|+|x_2|}v^{\frac{(|x_1|+ |x_2|)\cdot (|x_1|+|x_2|)}{2}}
K^{-1}_{|x_1|+|x_2|}\sigma(x_1x_2)^+.
\end{align*}
This finishes the proof.
\end{proof}

\begin{lem}\label{lem:deltar+}
For all  $x, x', x'' \in \mathfrak{f}$, let $r^+_i(x^+)=\iota^+ \circ r_i(x)$ and ${}_ir^+(x^+)=\iota^+ \circ {}_ir(x)$.
Then we have\\
$(\mathrm{i})\ \ \  r^+_i((x'x'')^+)=v^{i\cdot |x''|}t^{\langle |x''|,i\rangle-\langle i,|x''|\rangle}r^+_i(x^{'+})x^{''+}+x^{'+}r^+_i(x^{''+})$,\\
$(\mathrm{ii})\ \ {}_ir^+((x'x'')^+)={}_ir^+(x^{'+})x^{''+}+v^{i\cdot |x'|}t^{\langle i,|x'|\rangle-\langle |x'|,i\rangle}x^{'+}{}_ir^+(x^{''+})$,\\
$(\mathrm{iii})\ x^+F_i-F_ix^+=\cfrac{r^+_i(x^+)K_i-K'_i({}_ir^+(x^+))}{v_i-v_i^{-1}}$,\\
$(\mathrm{iv})\ \Delta(x^+)=x^+\otimes 1+\sum_{i}r^+_i(x^+)K_i\otimes E_i+\cdots
 =K_{|x|}\otimes x^++\sum_{i}E_iK_{|x|-i}\otimes{}_ir^+(x^+)+\cdots$.
\end{lem}

\begin{proof}
Statement $({\rm i})$ (resp. $({\rm ii})$) directly follows the definition of $\iota^+$ and $r_i$ (resp.$\ {}_ir$).

We now show $({\rm iii})$ by induction on $|x|$.
The case that $x=\theta_i$ is trivial.
Assume that $({\rm iii})$ holds for $x'$ and $x''$. Then we have
\begin{align*}
&\quad \quad (x'x'')^+F_i-F_i(x'x'')^+\\
&=x'^+(F_ix''^+ + \cfrac{r^+_i(x''^+)K_i-K'_i({}_ir^+(x''^+))}{v_i-v_i^{-1}})-F_ix'^+x''^+\\
&=\cfrac{r^+_i(x'^+)K_i-K'_i({}_ir^+(x'^+))}{v_i-v_i^{-1}}x''^+ +
x'^+\cfrac{r^+_i(x''^+)K_i-K'_i({}_ir^+(x''^+))}{v_i-v_i^{-1}}\\
&=\cfrac{r^+_i(x'^+)K_ix''^++x'^+r^+_i(x''^+)K_i}{v_i-v_i^{-1}}-
\cfrac{K'_i({}_ir^+(x'^+))x''^++x'^+K'_i({}_ir^+(x''^+))}{v_i-v_i^{-1}}\\
&=\cfrac{r^+_i((x'x'')^+)K_i+K'_i({}_ir^+((x'x'')^+))}{v_i-v_i^{-1}}.
\end{align*}
This proves $({\rm iii})$.

By $({\rm i})$ and $({\rm ii})$, statement $({\rm iv})$ can be shown by induction on $|x|$.
\end{proof}

\begin{lem}\label{lem:deltar-}
For all  $y, y', y'' \in \mathfrak{f},$ let $r^-_i(x^-)=\iota^- \circ r_i(x)$ and ${}_ir^-(x^-)=\iota^- \circ {}_ir(x)$.
Then we have\\
$(\mathrm{i})\ \ \ r^-_i((y'y'')^-)=v^{i\cdot |y''|}t^{\langle i,|y''|\rangle-\langle |y''|,i\rangle}r^-_i(y'^-)y''^- +y'^-r^-_i(y''^-)$,\\
$(\mathrm{ii})\ \  {}_ir^-((y'y'')^-)={}_ir^-(y'^-)y''^- +v^{i\cdot |y'|}t^{\langle |y'|,i\rangle-\langle i,|y'|\rangle}y'^-{}_ir(y''^-)$,\\
$(\mathrm{iii})\ E_iy^- -y^-E_i=\cfrac{K_i({}_ir^-(y^-))-r^-_i(y^-)K'_i}{v_i-v_i^{-1}}$,\\
$(\mathrm{iv})\ \Delta(y^-)=y^-\otimes K'_{|y|}+\sum_{i}{}_ir^-(y^-)\otimes F_iK'_{|y|-i}+\cdots
=1\otimes y^-+\sum_{i}F_i\otimes r^-_i(y^-)K'_i+\cdots$.
\end{lem}
The proof of this lemma is similar to that of Lemma \ref{lem:deltar+}.

\subsection{The quasi-$\mathcal{R}$-matrix $\Theta$}
In this section, we shall simply write $\mathbf{U}$ instead of  $U_{v,t}$.

We define a bar involution $\bar{\cdot}:\mathbf{U}\rightarrow \mathbf{U}$ such that
\begin{equation*}
\begin{split}
&\bar{E_i}=E_i,\ \bar{F_i}=F_i,\ \bar{K_i}=K'_i,\ \bar{K'_i}=K_i;\\
&\overline{px}=\bar{p}\cdot\bar{x},\ \forall p\in \mathbb{Q}(v,t), x\in \mathbf{U}.
\end{split}
\end{equation*}

Let $\bar{\cdot}:\mathbf{U}\otimes \mathbf{U}\rightarrow \mathbf{U}\otimes \mathbf{U}$
be the $\mathbb{Q}(t)$-algebra homomorphism given by $\bar{\cdot}\otimes \bar{\cdot }$
and $\overline{\Delta}:\mathbf{U}\rightarrow \mathbf{U}\otimes \mathbf{U}$
the $\mathbb{Q}(t)$-algebra homomorphism given by $\overline{\Delta}(x)=\overline{\Delta (\bar{x})}$.
Thus, we have
\begin{equation}\label{eq:bardelta}
\begin{split}
&\overline{\Delta}(E_i)=E_i\otimes 1+K'_i\otimes E_i,\quad \overline{\Delta}(K_i^{\pm 1})=K_i^{\pm 1}\otimes K_i^{\pm 1},\\
&\overline{\Delta}(F_i)=1\otimes F_i+F_i\otimes K_i,\quad \overline{\Delta}(K_i^{'\pm 1})=K_i^{'\pm 1}\otimes K_i^{'\pm 1}.
\end{split}
\end{equation}

Let $(\mathbf{U}\otimes \mathbf{U})^{\wedge}$ be the completion of the vector space $\mathbf{U}\otimes \mathbf{U}$
with respect to the descending sequence of vector spaces
$$
\mathcal{H}_N=(\mathbf{U}^+\mathbf{U}^0(\sum_{{\rm tr} \nu \geq N}\mathbf{U}^-_{\nu}))\otimes \mathbf{U}
+ \mathbf{U}\otimes (\mathbf{U}^-\mathbf{U}^0(\sum_{{\rm tr} \nu \geq N}\mathbf{U}^+_{\nu}))
$$
for $N=1,2,\cdots $ .

We set
\begin{equation*}
\{i,j\}=v^{i\cdot j}t^{\langle j, i\rangle - \langle i, j\rangle},\quad \forall i,j \in I,
\end{equation*}
which is a multiplicative bilinear form on $\mathbb{Z}[I]\times \mathbb{Z}[I]$.
\begin{lem}\cite[Proposition 4]{FX}\label{lem:hopfpair}
Let $\mathbf{U}_{\geq0}$ (resp. $\mathbf{U}_{\leq0}$) be the subalgebra of $\mathbf{U}$
generated by $E_i$ and $K_i$ (resp. $F_i$ and $K'_i$) for all $i$ in $I$.
We denote $K_{-\mu}$(resp. $K'_{-\mu}$) by $K^{-1}_{\mu}$(resp. $K'^{-1}_{\mu}$)
for all $\mu \in \mathbb{N}[I]$.
There is a skew-Hopf pairing
$(,)_{\phi}:\mathbf{U}_{\geq0} \times \mathbf{U}_{\leq0}\rightarrow \mathbb{Q}(v,t)$ such that\\
$(\mathrm{i})\ \ (1,1)_{\phi}=1,\\
(\mathrm{ii})\ (E_i, F_j)_{\phi}=\delta_{ij}(v_i^{-1}-v_i)^{-1},\quad \forall i,j\in I,\\
(\mathrm{iii})(K_{\mu}x, K'_{\nu}y)_{\phi}=\{\mu, \nu\}\{\mu, |y|\}\{|x|, \nu\}(x,y)_{\phi},\
\forall \mu,\nu\in \mathbb{Z}[I],\ x\in \mathbf{U}^+,\ y\in \mathbf{U}^-.$
\end{lem}

For any homogenous elements $x\in \mathbf{U}^+, y \in \mathbf{U}^-$, we have
\begin{equation}\label{eq:innerproduct}
(xK_{\mu},yK'_{\nu})_{\phi}=(x,y)_{\phi}(K_{\mu},K'_{\nu})_{\phi},\quad \forall \mu, \nu \in \mathbb{Z}[I].
\end{equation}
By Lemma \ref{lem:deltar+}$(\mathrm{iv})$ and Lemma \ref{lem:deltar-}$(\mathrm{iv})$, we have
\begin{equation*}
\begin{split}
(xK_{\mu},yK'_{\nu})_{\phi}&=(\Delta(x)\Delta(K_{\mu}),y\otimes K'_{\nu})_{\phi}\\
&=(xK_{\mu},y)_{\phi}(K_{\mu},K'_{\nu})_{\phi}\\
&=(x\otimes K_{\mu},\Delta^{op}(y))_{\phi}(K_{\mu},K'_{\nu})_{\phi}\\
&=(x,y)_{\phi}(K_{\mu},K'_{\nu})_{\phi}.
\end{split}
\end{equation*}
This proves \eqref{eq:innerproduct}.
\begin{lem}\label{lem:phiprop}
For all $x \in \mathbf{U}^+$ and $y \in  \mathbf{U}^-$, we have\\
$(\mathrm{i})\ \ \ (x,F_iy)_{\phi}=(v_i^{-1}-v_i)^{-1}({}_ir^+(x),y)_{\phi}$,\\
$(\mathrm{ii})\ \  (x,yF_i)_{\phi}=(v_i^{-1}-v_i)^{-1}(r_i^+(x),y)_{\phi}$,\\
$(\mathrm{iii})\ (E_ix,y)_{\phi}=(v_i^{-1}-v_i)^{-1}(x,{}_ir^-(y))_{\phi}$,\\
$(\mathrm{iv})\ (xE_i,y)_{\phi}=(v_i^{-1}-v_i)^{-1}(x,r_i^-(y))_{\phi}$.
\end{lem}
\begin{proof}
The proofs of the four equations are similar. We shall only show $(\mathrm{i})$ and left others to readers.

By the definition of skew-Hopf pairing $(,)_{\phi}$ in \cite[Scetion 2.2]{Xiao1}, Lemma \ref{lem:deltar+}$(\mathrm{iv})$
and Lemma \ref{lem:hopfpair}$(\mathrm{iii})$, we have
\begin{equation*}
\begin{split}
(x,F_iy)_{\phi}&=(\Delta(x),F_i\otimes y)_{\phi}\\
&=(E_iK_{|x|-i}\otimes {}_ir^+(x), F_i\otimes y)_{\phi}\\
&=(E_iK_{|x|-i}, F_i)_{\phi}({}_ir^+(x), y)_{\phi}\\
&=(E_i,F_i)_{\phi}({}_ir^+(x), y)_{\phi}\\
&=(v^{-1}_i-v_i)^{-1}({}_ir^+(x), y)_{\phi}.
\end{split}
\end{equation*}
\end{proof}

\begin{lem}\label{lem:hopfpairprop}
For all $x, y\in \mathfrak{f}$, we have $(x^+,y^-)_{\phi}=(\sigma^+(x^+),\sigma^-(y^-))_{\phi}$.
\end{lem}
\begin{proof}
It is straightforward to check it when
$x=\theta_i$ and $y=\theta_j$ for some $i,j\in I$.
Assume that the lemma holds for $x_1$ and $x_2$. We shall show that it holds for $x=x_1x_2$.

Let  $y\in \mathfrak{f}$ and  $r(y)=\sum y_1 \otimes y_2$ with $y_1, y_2$ homogeneous.
Then we have
$$
\Delta(y^-)=\sum v^{-|y_1|\cdot |y_2|}t^{\langle|y_1|,|y_2|\rangle-
\langle|y_2|,|y_1|\rangle}y^-_2\otimes K'_{|y_2|}y^-_1.
$$
By Lemma \ref{lem:sigma}, $r(\sigma(y))=\sum t^{\langle|y_2|,|y_1|\rangle-
\langle|y_1|,|y_2|\rangle}\sigma(y_2)\otimes \sigma(y_1)$.
Then we have
\begin{equation}\label{eq:sigmay-}
\Delta(\sigma(y)^-)=\sum v^{-|y_1|\cdot|y_2|}\sigma(y_1)^-\otimes K'_{|y_1|}\sigma(y_2)^-.
\end{equation}
By \eqref{eq:sigma-f}, \eqref{eq:sigmay-} and Lemma \ref{lem:hopfpair}$(\mathrm{iii})$, we have
\begin{equation}\label{eq:sigmaphir}
\begin{split}
&(\sigma^+(x_1^+x_2^+), \sigma^-(y^-))_{\phi}\\
=&t^{\langle|x_2|,|x_1|\rangle-\langle|x_1|,|x_2|\rangle}(\sigma^+(x_2^+)\sigma^+(x_1^+), \sigma^-(y^-))_{\phi}\\
=&t^{\langle|x_2|,|x_1|\rangle-\langle|x_1|,|x_2|\rangle}(\sigma^+(x_2^+)\otimes\sigma^+(x_1^+), \Delta^{op}(\sigma^-(y^-)))_{\phi}\\
=&\sum t^{\langle|x_2|,|x_1|\rangle-\langle|x_1|,|x_2|\rangle}v^{-|y_1|\cdot|y_2|}
(\sigma^+(x_2^+)\otimes\sigma^+(x_1^+),K'_{|y_1|}\sigma^-(y_2^-)\otimes\sigma^-(y_1^-))_{\phi}\\
=&\sum t^{\langle|x_2|,|x_1|\rangle-\langle|x_1|,|x_2|\rangle}v^{-|y_1|\cdot|y_2|}
(\sigma^+(x_2^+),K'_{|y_1|}\sigma^-(y_2^-))_{\phi}(\sigma^+(x_1^+),\sigma^-(y_1^-))_{\phi}\\
=&\sum (\sigma^+(x_2^+),\sigma^-(y_2^-))_{\phi}(\sigma^+(x_1^+),\sigma^-(y_1^-))_{\phi}.
\end{split}
\end{equation}

Similarly, we have
\begin{equation}\label{eq:sigmaphil}
\begin{split}
&(x_1^+x_2^+, y^-)_{\phi}=(x_1^+\otimes x_2^+, \Delta^{op}(y^-))_{\phi}\\
=&\sum v^{-|y_1|\cdot|y_2|}t^{\langle|y_1|,|y_2|\rangle-\langle|y_2|,|y_1|\rangle}
(x_1^+\otimes x_2^+, K'_{|y_2|}y_1^-\otimes y_2^-)_{\phi}\\
=&\sum v^{-|y_1|\cdot|y_2|}t^{\langle|y_1|,|y_2|\rangle-\langle|y_2|,|y_1|\rangle}
(x_1^+, K'_{|y_2|}y_1^-)_{\phi}(x_2^+,y_2^-)_{\phi}\\
=&\sum (x_1^+,y_1^-)_{\phi}(x_2^+,y_2^-)_{\phi}.
\end{split}
\end{equation}

By \eqref{eq:sigmaphir} and \eqref{eq:sigmaphil}, we have
$$(x_1^+x_2^+, y^-)_{\phi}=(\sigma^+(x_1^+x_2^+), \sigma^-(y^-))_{\phi}.$$
This finishes the proof.
\end{proof}
By the relation $(R2)$ in Section \ref{sec:presentation} ,
the subalgebra $\mathbf{U}^+$ has the following decomposition
$$\mathbf{U}^+ = \bigoplus_{\mu\in \mathbb{N}[I]}\mathbf{U}^+_{\mu},$$
where
$$\mathbf{U}^+_{\mu}=\{u\in \mathbf{U}^+ |
K_iu=v^{i\cdot |u|}t^{\langle|u|,i\rangle-\langle i,|u|\rangle}uK_i,\
K'_iu=v^{-i\cdot |u|}t^{\langle|u|,i\rangle-\langle i,|u|\rangle}uK'_i,\ \forall i \in I \}.$$
The weight space $\mathbf{U}^+_{\mu}$ is spanned by all the monomials $E_{i_1}\cdots E_{i_l}$
with grading $\mu$.

Similarly, the subalgebra $\mathbf{U}^-$ has a decomposition
$\mathbf{U}^- = \bigoplus_{\mu\in \mathbb{N}[I]}\mathbf{U}^-_{-\mu}$
and the spaces $\mathbf{U}^+_{\mu}$ and $\mathbf{U}^-_{-\mu}$
are nondegenerately paired under the skew-Hopf pairing $(,)_{\phi}$.
Then we may select a basis $\mathbf{B}$ of $\mathbf{U}^-$ such that
 $\mathbf{B}_{\mu}=\mathbf{B}\cap\mathbf{U}^-_{-\mu}$.
Let $\{b^*|b\in \mathbf{B}_{\mu}\}$ be the basis of $\mathbf{U}^+_{\mu}$
dual to $\mathbf{B}_{\mu}$ under $(,)_{\phi}$.

\begin{lem}\label{lem:delta}
Let $x\in \mathbf{U}^+_{\lambda}$ and $y\in \mathbf{U}^-_{-\lambda}$
for any $\lambda \in \mathbb{N}[I]$. Then,\\
$(\mathrm{i})\ \ \Delta(x)=\sum_{0\leq \mu \leq \lambda}
\sum_{b\in \mathbf{B}_{\mu},b'\in \mathbf{B}_{\lambda-\mu}}
  (x,b'b)_{\phi}b'^*K_{\mu}\otimes b^*,$\\
$(\mathrm{ii})\ \Delta(y)=\sum_{0\leq \mu \leq \lambda}
\sum_{b\in \mathbf{B}_{\mu},b'\in \mathbf{B}_{\lambda-\mu}}
  (b'^*b^*,y)_{\phi}b\otimes b'K'_{\mu}.$
\end{lem}
\begin{proof}
The proofs of $(\mathrm{i})$ and $(\mathrm{ii})$ are similar. We shall only show $(\mathrm{i})$.

As $x\in \mathbf{U}^+_{\lambda}$,
we have $\Delta(x)\in \sum_{0\leq \mu \leq \lambda}
\mathbf{U}^+_{\lambda-\mu}K_{\mu}\otimes \mathbf{U}^+_{\mu}$.
Let $h^{\mu}_{b,b'}\in \mathbb{Q}(v,t)$ be such that
$$
\Delta(x)=\sum_{0\leq \mu \leq \lambda}
\sum_{b\in \mathbf{B}_{\mu},b'\in \mathbf{B}_{\lambda-\mu}}
h^{\mu}_{b,b'}b'^*K_{\mu}\otimes b^*.
$$
Then for all $b_1\in \mathbf{B}_{\lambda-\mu}, b_2\in \mathbf{B}_{\mu}$  and $\mu$,
we have
\begin{equation*}
\begin{split}
(x,b_1b_2)_{\phi}&=(\Delta(x),b_1\otimes b_2)_{\phi}\\
&=\sum_{0\leq \mu \leq \lambda}\sum_{b\in \mathbf{B}_{\mu},b'\in \mathbf{B}_{\lambda-\mu}}
h^{\mu}_{b,b'}(b'^*K_{\mu}\otimes b^*,b_1\otimes b_2)_{\phi}\\
&=\sum_{0\leq \mu \leq \lambda}\sum_{b\in \mathbf{B}_{\mu},b'\in \mathbf{B}_{\lambda-\mu}}
h^{\mu}_{b,b'}(b'^*K_{\mu},b_1)_{\phi}( b^*,b_2)_{\phi}\\
&=h^{\mu}_{b_2,b_1}.
\end{split}
\end{equation*}
This finishes the proof.
\end{proof}

For each $x\in \mathbf{U}^+_{\mu}$ and $y\in \mathbf{U}^-_{-\mu}$, we have
\begin{equation}\label{eq:pair}
x=\sum_{b\in \mathbf{B}_{\mu}}(x, b)_{\phi}b^*,\quad y=\sum_{b\in \mathbf{B}_{\mu}}(b^*, y)_{\phi}b.
\end{equation}

For $\mu\in \mathbb{N}[I]$, we define
\begin{equation*}
 \Theta_{\mu}=\sum_{b\in \mathbf{B}_{\mu}}b\otimes b^{*}\in \mathbf{U}^-_{-\mu}\otimes \mathbf{U}^+_{\mu}.
\end{equation*}
Set $\Theta_{\mu}=0$ if $\mu \notin \mathbb{N}[I]$.
\begin{lem}\label{lem:quasiRmatrix}
For all $i\in I, \mu \in \mathbb{N}[I]$, we have\\
$(\mathrm{i})\ \ \ (K_i\otimes K_i)\Theta_{\mu}=\Theta_{\mu}(K_i\otimes K_i)$,\\
$(\mathrm{ii})\ \  (K'_i\otimes K'_i)\Theta_{\mu}=\Theta_{\mu}(K'_i\otimes K'_i)$,\\
$(\mathrm{iii})\ (E_i\otimes 1)\Theta_{\mu}+(K_i\otimes E_i)\Theta_{\mu-i}=
 \Theta_{\mu}(E_i\otimes 1)+\Theta_{\mu-i}(K'_i\otimes E_i)$,\\
$(\mathrm{iv})\ (1\otimes F_i)\Theta_{\mu}+(F_i\otimes K'_i)\Theta_{\mu-i}=
 \Theta_{\mu}(1\otimes F_i)+\Theta_{\mu-i}(F_i\otimes K_i)$.
\end{lem}

\begin{proof}
The first two are easy to check.
We shall show $(\mathrm{iii})$ and $(\mathrm{iv})$ can be shown similarly.
The proof of $(\mathrm{iii})$ goes in a similar way as that for Lemma 4.10 in \cite{BW}.
For the convenience of the readers, we present it here.
By Lemma \ref{lem:deltar-}$(\mathrm{iii})$ ,
Lemma \ref{lem:phiprop}$(\mathrm{iii})$-$(\mathrm{iv})$ and \eqref{eq:pair},
we have
\begin{equation*}
\begin{split}
&(E_i\otimes 1)\Theta_{\mu}-\Theta_{\mu}(E_i\otimes 1)\\
=&\sum_{b\in \mathbf{B}_{\mu}}(E_ib-bE_i)\otimes b^*\\
=&(v_i-v^{-1}_i)^{-1}\sum_{b\in \mathbf{B}_{\mu}}(K_i({}_ir^-(b))-r_i^{-1}(b)K'_i)\otimes b^*\\
=&(v_i-v^{-1}_i)^{-1}(\sum_{b\in \mathbf{B}_{\mu}}K_i\sum_{b'\in \mathbf{B}_{\mu-i}}(b'^*,{}_ir^-(b))_{\phi}b'\otimes b^*
-\sum_{b\in \mathbf{B}_{\mu}}\sum_{b'\in \mathbf{B}_{\mu-i}}(b'^*,r_i^-(b))_{\phi}b'K'_i\otimes b^*)\\
=&-\sum_{b\in \mathbf{B}_{\mu}}K_i\sum_{b'\in \mathbf{B}_{\mu-i}}(E_ib'^*,b)_{\phi}b'\otimes b^*
+\sum_{b\in \mathbf{B}_{\mu}}\sum_{b'\in \mathbf{B}_{\mu-i}}(b'^*E_i,b)_{\phi}b'K'_i\otimes b^*\\
=&\sum_{b'\in \mathbf{B}_{\mu-i}}b'K'_i\otimes \sum_{b\in \mathbf{B}_{\mu}}(b'^*E_i,b)_{\phi} b^*
-\sum_{b'\in \mathbf{B}_{\mu-i}}K_ib'\otimes\sum_{b\in \mathbf{B}_{\mu}}(E_ib'^*,b)_{\phi} b^*\\
=&\sum_{b'\in \mathbf{B}_{\mu-i}}b'K'_i\otimes b'^*E_i-\sum_{b'\in \mathbf{B}_{\mu-i}}K_ib'\otimes E_ib'^*\\
=&\Theta_{\mu-i}(K'_i\otimes E_i)-(K_i\otimes E_i)\Theta_{\mu-i}.
\end{split}
\end{equation*}
This finishes the proof.
\end{proof}

\begin{prop}\label{prop:quasiR}
Let $\Theta_{0}=1\otimes 1$  and
$\Theta=\sum_{\nu \in \mathbb{N}[I]}\Theta_{\nu}\in(\mathbf{U}\otimes \mathbf{U})^{\wedge}$.
Then we have $\Delta(u)\Theta=\Theta\overline{\Delta}(u)$ for all
 $u\in \mathbf{U}$ $($where this identity is in $(\mathbf{U}\otimes \mathbf{U})^{\wedge})$.
\end{prop}
This proposition follows from Lemma \ref{lem:quasiRmatrix}.
The element $\Theta$ defined in this proposition is called the $quasi$-$\mathcal{R}$-$matrix$.

\begin{cor}\label{cor:bartheta}
We have $\Theta \overline{\Theta}=\overline{\Theta}\Theta=1\otimes 1$
$($equality in $(\mathbf{U}\otimes \mathbf{U})^{\wedge})$.
\end{cor}
The proof is similar to those for Corollary 4.1.3 in \cite{Lusztigbook}.

We define $(,)_{\bar{\phi}}:\mathbf{U}^+ \times \mathbf{U}^-\rightarrow \mathbb{Q}(v,t)$ by
\begin{equation}\label{eq:barhopfpair}
(x,y)_{\bar{\phi}}=\overline{(\bar{x}, \bar{y})_{\phi}},
\quad \forall x\in \mathbf{U}^+ ,y\in\mathbf{U}^-.
\end{equation}
It satisfies that $(1, 1)_{\bar{\phi}}=1$ and
$(E_i, F_j)_{\bar{\phi}}=\delta_{ij}(v_i-v^{-1}_i)^{-1}$.

\begin{lem} \label{lem:barhopfpair}
$(x^+,y^-)_{\bar{\phi}}=(-1)^{{\rm tr} |x|}v^{-|x|\cdot |y|/2}v_{-|x|}(x^+,\sigma^-(y^-))_{\phi},\ \forall x, y \in \mathfrak{f}$.
\end{lem}
\begin{proof}
It is straightforward to check it when
$x=\theta_i$ and $y=\theta_j$ for some $i,j\in I$.

Let  $x\in \mathfrak{f}$ and  $r(x)=\sum x_1 \otimes x_2$ with $x_1, x_2$ homogeneous.
Assume that the lemma holds for $y_1$ and $y_2$. We shall show that it holds for $y=y_1y_2$.

By Lemma \ref{lem:rbar}, we have
$\Delta(\bar{x^+})=\sum v^{|x_1|\cdot|x_2|}t^{\langle |x_2|, |x_1|\rangle-\langle |x_1|, |x_2|\rangle}
\bar{x_2^+}K_{|x_1|}\otimes \bar{x_1^+}$. Then,
\begin{equation*}
\begin{split}
(\bar{x^+},\bar{y_1^-}\bar{y_2^-})_{\phi}=&(\Delta(\bar{x^+}), \bar{y_1^-}\otimes \bar{y_2^-})_{\phi}\\
=&\sum v^{|x_1|\cdot|x_2|}t^{\langle |x_2|, |x_1|\rangle-\langle |x_1|, |x_2|\rangle}
(\bar{x_2^+}K_{|x_1|}, \bar{y_1^-})_{\phi}(\bar{x_1^+}, \bar{y_2^-})_{\phi}\\
=&\sum v^{|x_1|\cdot|x_2|}t^{\langle |x_2|, |x_1|\rangle-\langle |x_1|, |x_2|\rangle}
(\bar{x_2^+}, \bar{y_1^-})_{\phi}(\bar{x_1^+}, \bar{y_2^-})_{\phi}.
\end{split}
\end{equation*}
By \eqref{eq:barhopfpair}, we have
\begin{equation}\label{eq:barphil}
\begin{split}
(x^+,y_1^-y_2^-)_{\bar{\phi}}
=&\sum v^{-|x_1|\cdot|x_2|}t^{\langle |x_2|, |x_1|\rangle-\langle |x_1|, |x_2|\rangle}
(x_2^+, y_1^-)_{\bar{\phi}}(x_1^+, y_2^-)_{\bar{\phi}}\\
=&\sum (-1)^{\mathrm{tr}(|x_1|+|x_2|)}v^{-(|x_1|\cdot|x_2|+(|x_1|\cdot|y_2|+|x_2|\cdot|y_1|)/2)}v_{-(|x_1|+|x_2|)}\\
&\quad\  t^{\langle |x_2|, |x_1|\rangle-\langle |x_1|, |x_2|\rangle}
(x_1^+,\sigma^-(y_2^-))_{\phi}(x_2^+,\sigma^-(y_1^-))_{\phi}.
\end{split}
\end{equation}
On the other hand,
\begin{equation}\label{eq:barphir}
\begin{split}
(x^+,\sigma^-(y_1^-y_2^-))_{\phi}=&t^{\langle |y_1|, |y_2|\rangle-\langle |y_2|, |y_1|\rangle}
(\Delta(x^+),\sigma^-(y_2^-)\sigma^-(y_1^-))_{\phi}\\
=&t^{\langle |y_1|, |y_2|\rangle-\langle |y_2|, |y_1|\rangle}
\sum (x^+_1K_{|x_2|}\otimes x^+_2,\sigma^-(y_2^-)\otimes \sigma^-(y_1^-))_{\phi}\\
=&t^{\langle |y_1|, |y_2|\rangle-\langle |y_2|, |y_1|\rangle}
\sum (x^+_1,\sigma^-(y_2^-))_{\phi}\sum (x^+_2,\sigma^-(y_1^-))_{\phi}.
\end{split}
\end{equation}
This lemma follows from \eqref{eq:barphil} and \eqref{eq:barphir} with $|x_1|=|y_2|$ and $|x_2|=|y_1|$.
\end{proof}

While $\bar\Theta$ can be evaluated easily, it will be more convenient
to have the following alternate description of $\bar\Theta$
using the property of $(,)_{\bar{\phi}}$.

\begin{lem}\label{lem:barquasiR}
With the same notations as in Proposition \ref{prop:quasiR},
$\bar\Theta=\sum_\nu \bar\Theta_\nu$ is given by
\begin{equation*}
\bar \Theta_\nu=(-1)^{\mathrm{tr}\nu}v^{\frac{\nu\cdot\nu}2}v_{-\nu}
\sum_{b\in \mathbf{B}_\nu} b\otimes \sigma^+(b^*)\in \mathbf U_{-\nu}^-\otimes \mathbf U_\nu^+.
\end{equation*}
\end{lem}
\begin{proof}
Since $\Theta$ is independent of the choice of basis, we let
\begin{equation*}
\Theta_\nu=\sum_{b\in \mathbf{B}_\nu} \bar{b}\otimes \bar{b}^*,\quad \forall \nu \in \mathbb{N}[I],
\end{equation*}
where $\bar{b} \in \bar{\mathbf{B}}=\{\bar{b}|b\in \mathbf{B} \}$.
Then
\begin{equation}\label{eq:barthetanu}
\bar\Theta_\nu=\sum_{b\in \mathbf{B}_\nu} \overline{\bar{b}}\otimes \overline{(\bar{b}^*)},
\quad \forall \nu \in \mathbb{N}[I].
\end{equation}
Note that $\overline{\bar{b}}=b$.

We show the relation between $\overline{\bar{b}^*}$ and $\sigma^+(b^*)$.
There exists an element $\bar{b'}\in \bar{\mathbf{B}_{\nu}}$ such that
$(\bar{b}^*,\bar{b'})_{\phi}=\delta_{b,b'}$. By \eqref{eq:barhopfpair} and Lemma \ref{lem:barhopfpair}, we have
\begin{equation*}
\delta_{b,b'}=\overline{(\bar{b}^*, \bar{b'})_{\phi}}=(\bar{\bar{b}^*},b')_{\bar{\phi}}
=(-1)^{\rm{tr} \nu}v^{-\nu\cdot \nu/2}v_{-\nu}(\bar{\bar{b}^*},\sigma^-(b'))_{\phi}.
\end{equation*}
Therefore, we have
\begin{equation}\label{eq:barbarb}
\bar{\bar{b}^*}=(-1)^{\rm{tr} \nu}v^{\nu\cdot \nu/2}v_{\nu}\sigma^-(b)^*.
\end{equation}
By Lemma \ref{lem:hopfpairprop}, we have
\begin{equation}\label{eq:sigmab*}
\sigma^-(b)^*=\sigma^+(b^*).
\end{equation}
The Lemma follows from \eqref{eq:barthetanu}, \eqref{eq:barbarb} and \eqref{eq:sigmab*}.
\end{proof}

\section{The $\mathcal{R}$-matrix for two-parameter quantum algebras}\label{sec:rmatrix}
\subsection{The module of $U_{v,t}$}
A $U_{v,t}$-module $M$ is called a $weight\ module$ if it admits a decomposition $M=\bigoplus_{\lambda\in \mathbb{N}[I]}M_{\lambda}$ of
vector spaces such that
\begin{equation*}
M_{\lambda}=\{m\in M|K_i\cdot m=v^{i\cdot \lambda}c_{i,\lambda}m,\ K'_i\cdot m=v^{-i\cdot \lambda}c_{i,\lambda}m,\ \forall i\in I\},
\end{equation*}
where
\begin{equation*}
c_{i,\lambda}=t^{\langle\lambda,i\rangle-\langle i,\lambda\rangle}.
\end{equation*}
For any $m\in M_{\lambda}$, we denote by $|m|=\lambda$.

For all $m\in \mathbb{Q}(v,t),$ we define
$$u\cdot m=\varepsilon(u)m,\quad \forall u\in U_{v,t},$$
 where $\varepsilon$ is the counit of $U_{v,t}$.
 Then, $\mathbb{Q}(v,t)$ is a trivial module of $U_{v,t}$.

Let $M$ be $U_{v,t}$-module and $M^{*}=Hom_{\mathbb{Q}(v,t)}(M,\mathbb{Q}(v,t))$.
We define $u\cdot n^* \in M^*$ by
\begin{equation}\label{eq:uaction}
u\cdot n^*(m)=n^*(S(u)\cdot m), \quad \forall u \in U_{v,t}, n^* \in M^{*}, m\in M,
\end{equation}
then $M^{*}$ is also a $U_{v,t}$-module.

 For any $U_{v,t}$-modules $M$ and $N$, we can construct the $U_{v,t}$-module
$M\otimes N=M\otimes_{\mathbb{Q}(v,t)} N$ via the coproduct.
In particular, we have $U_{v,t}$-modules
$M^*\otimes M$ and $M\otimes M^*$.

For any $i\in I, \lambda \in \mathbb{N}[I]$, we denote by
\begin{equation*}
v_{-\lambda}=v^{-1}_{\lambda},\quad {\rm and}\quad c_{i,-\lambda}=c^{-1}_{i,\lambda}.
 \end{equation*}
\begin{lem}\label{lem:fourmaps}
Fix a $U_{v,t}$-module M.

$(1)$\quad Let $\mathrm{ev}: M^*\otimes M \rightarrow \mathbb{Q}(v,t)$
be the $\mathbb{Q}(v,t)$-linear map defined by
$$m^*\otimes n\mapsto m^*(n),\quad \forall m^*\in M^*,n \in M.$$
Then $\mathrm{ev}$ is a $U_{v,t}$-module epimorphism.

$(2)$\quad Let $\mathrm{qtr}: M\otimes M^*\rightarrow \mathbb{Q}(v,t)$
be the $\mathbb{Q}(v,t)$-linear map defined by
$$m\otimes n^*\mapsto\ v_{-|m|}^2\ n^*(m),\quad \forall m\in M,n^* \in M^*.$$
Then $\mathrm{qtr}$ is a $U_{v,t}$-module epimorphism.

$(3)$\quad Let $\mathrm{coev}: \mathbb{Q}(v,t) \rightarrow M^*\otimes M$
be the $\mathbb{Q}(v,t)$-linear map defined by
$$1\mapsto \sum_{w\in B}\ v_{|w|}^2\ w^*\otimes w$$
for some homogeneous $\mathbb{Q}(v,t)$-basis B of M.
Then $\mathrm{coev}$ is a $U_{v,t}$-module monomorphism.

$(4)$\quad Let $\mathrm{coqtr}: \mathbb{Q}(v,t) \rightarrow M\otimes M^*$
be the $\mathbb{Q}(v,t)$-linear map defined by
$$1\mapsto \sum_{w\in B}w\otimes w^*$$
for some homogeneous $\mathbb{Q}(v,t)$-basis B of M.
Then $\mathrm{coqtr}$ is a $U_{v,t}$-module monomorphism.
\end{lem}
\begin{proof}
The surjectivity and injectivity of the maps are clear.
We shall show that these maps preserve the action of generators of $U_{v,t}$.
It is straightforward to verify these maps holds for $K_i$ and $K'_i$.
It's enough to check that these maps preserve the action of $E_i$ and $F_i$ for all $i\in I$.
Show that
\begin{equation}
   {\rm ev}(\Delta(E_i)m^*\otimes n)={\rm ev}(\Delta(F_i)m^*\otimes n)=0,\quad  \forall m^*\in M^* ,n\in M,
\end{equation}
\begin{equation}\label{eq:qtrdelta}
{\rm qtr}(\Delta(E_i)m\otimes n^*)={\rm qtr}(\Delta(F_i)m\otimes n^*)=0,\quad  \forall m\in M, n^*\in M^*,
\end{equation}
\begin{equation}\label{eq:deltacoev}
\Delta(E_i)\sum_{w\in B} v_{|w|}^2 w^*\otimes w
=\Delta(F_i)\sum_{w\in B} v_{|w|}^2 w^*\otimes w=0,
\end{equation}
\begin{equation}
\Delta(E_i)\sum_{w\in B} w\otimes w^*=
\Delta(F_i)\sum_{w\in B} w\otimes w^*=0.
\end{equation}

We shall prove \eqref{eq:qtrdelta} and \eqref{eq:deltacoev} for the action of $E_i$,
 the remaining cases can be shown similarly.

First, we show ${\rm qtr}(\Delta(E_i)m\otimes n^*)=0$.
For homogenous elements $m\in M, n^*\in M^*$, we have
\begin{align*}
{\rm qtr}(\Delta(E_i)m\otimes n^*)&=
{\rm qtr}(E_i\cdot m\otimes n^*+K_i\cdot m\otimes E_i\cdot n^*)\\
&=v^2_{-|m|-i}n^*(E_i\cdot m)+v^2_{-|m|}(E_i\cdot n^*)(K_i\cdot m)\\
&=v^2_{-|m|-i}n^*(E_i\cdot m)+v^2_{-|m|}n^*(-K_i^{-1}E_iK_i\cdot m)\\
&=v^2_{-|m|-i}n^*(E_i\cdot m)-v^2_{-|m|-i}n^*(E_i\cdot m)=0.
\end{align*}

Next, we show that $\Delta(E_i)\sum_{w\in B} v_{|w|}^2 w^*\otimes w=0$.
Observe that $x=\sum m^*\otimes n=0$ if and only if
$x(m'):=\sum m^*(m')n=0$ for all $m'\in M$.
Let $x=\Delta(E_i)\sum_{w\in B} v_{|w|}^2 w^*\otimes w$. Then we have
\begin{equation*}
x(m)=\sum_{w\in B} v_{|w|}^2 (E_i\cdot w^*(m) w+(K_i\cdot w^*)(m) E_i\cdot w),\  \forall m\in M.
\end{equation*}
For any $w_0\in B$, we have
\begin{align*}
x(w_0)=&\sum_{w\in B} v_{|w|}^2
(w^*(-K_i^{-1}E_i\cdot w_0) w+w^*(K_i^{-1}\cdot w_0) E_i\cdot w)\\
=&-\sum_{w\in B}v_{|w|}^2v^{-i\cdot(|w_0|+i)}c_{i,|w_0|+i}^{-1}w^*(E_i\cdot w_0) w\\
&+\sum_{w\in B}v_{|w|}^2v^{-i\cdot|w_0|}c_{i,|w_0|}^{-1}w^*(w_0) E_i\cdot w\\
=&-v_{|w_0|+i}^2v^{-i\cdot(|w_0|+i)}c_{i,|w_0|}^{-1}E_i\cdot w_0
+v_{|w_0|}^2v^{-i\cdot |w_0|}c_{i,|w_0|}^{-1} E_i\cdot w_0\\
=&0.
\end{align*}
Hence, $x=0$.

This finishes the proof.
\end{proof}
\subsection{The $\mathcal R$-matrix of $U_{v,t}$-module}\label{subsection:Rmatrix}
We shall construct a $U_{v,t}$-module isomorphism $\mathcal{R}_{M,M'}:M\otimes M'\rightarrow M'\otimes M$
for any finite dimensional weight modules $M$ and $M'$,
by the method used by Jantzen \cite[Chap. 7]{Ja} for the quantum algebras $U_q(\mathfrak{g})$.

The map $\mathcal{R}_{M,M'}$ is the composite of three linear transformations $P, \tilde{f}, \Theta$ defined as follows.

Let $P:M\otimes M'\rightarrow M'\otimes M$
be the $\mathbb{Q}(v,t)$-linear bijection defined by
$$P(m\otimes m')=m'\otimes m,\quad \forall m\in M, m'\in M'.$$

Recall that $(,)_{\phi}$ is a skew-Hopf pairing defined in Lemma \ref{lem:hopfpair}.
For any $\lambda, \mu\in \mathbb{Z}[I]$, we define the map
$f:\mathbb{Z}[I] \times \mathbb{Z}[I] \rightarrow \mathbb{Q}(v,t)^\times$ by
\begin{equation}\label{eq:fdef}
f(\lambda, \mu)=(K_{\lambda},K'_{\mu})^{-1}_{\phi}.
\end{equation}
Then we have
\begin{equation}\label{eq:fprop}
\begin{split}
f(\lambda+\mu,\nu)&=f(\lambda,\nu)f(\mu,\nu)\\
f(\lambda,\mu+\nu)&=f(\lambda,\mu)f(\lambda,\nu)\\
f(\lambda,\mu)&=f(-\lambda,-\mu)\\
f(i,\mu)&=v^{-i\cdot \mu}c_{\mu,i}\\
f(\lambda,i)&=v^{-i\cdot \lambda}c_{i,\lambda},
\end{split}
\end{equation}
where $\nu$ is also in $\mathbb{Z}[I]$.

We define a bijective linear map
$\tilde{f}:M\otimes M' \rightarrow M\otimes M'$ by
\begin{equation}\label{eq:tildef}
\tilde{f}(m\otimes m')=f(\lambda,\mu)m\otimes m',
\quad \forall m\in M_{\lambda}, m'\in M'_{\mu}, \lambda, \mu \in \mathbb{Z}[I].
\end{equation}

Recall the definition of $\Theta$ from Proposition \ref{prop:quasiR}.
The linear transformation $\Theta=\Theta_{M,M'}:M\otimes M'\rightarrow M\otimes M'$ is well-defined.

\begin{prop}\label{prop:Rmat}
For any  finite dimensional weight module $M, M'$ of $U_{v,t}$,
we define $\mathcal{R}=\mathcal{R}_{M,M'}: M\otimes M'\rightarrow M'\otimes M$
by $\mathcal{R}=\Theta\circ \tilde{f} \circ P$.
Then $\mathcal R$ is a $U_{v,t}$-module isomorphism.
\end{prop}
\begin{proof}
By Corollary \ref{cor:bartheta}, $\mathcal R$ is invertible transformation. Then,
we shall show that $\mathcal R$ is a $U_{v,t}$-module homomorphism, i.e.,
$$\Delta(u)\mathcal{R}(m\otimes m')=\mathcal{R}(\Delta(u)m\otimes m'),
\quad \forall u\in U_{v,t}, m\in M, m'\in M'.$$
By Proposition \ref{prop:quasiR}, we have
\begin{equation*}
\Delta(u)\mathcal R(m\otimes m')
=\Theta\bar\Delta(u)\tilde{f}\circ P(m\otimes m')
=\Theta(f(|m'|,|m|)\bar\Delta(u)(m'\otimes m)).
\end{equation*}
So it is suffices to show
\begin{equation*}
\tilde{f}\circ P(\Delta(u)m\otimes m'))=
f(|m'|,|m|)\bar\Delta(u)(m'\otimes m)
\end{equation*}
for all $u\in U_{v,t}$. Hence it is enough to show that this equality holds
all generators of $U_{v,t}$. For $u=K_\nu, K'_\nu$, this is straightforward.
The cases $u=E_i$ and $u=F_i$ are similar, so we shall prove the first case.

By \eqref{eq:bardelta} and \eqref{eq:fprop}, we have
\begin{align*}
&\tilde{f}\circ P(\Delta(E_i)m\otimes m')\\
= &f(|m'|,i+|m|)m'\otimes E_im+f(i+|m'|,|m|) E_im'\otimes K_im\\
=&f(|m'|,|m|)v^{-i\cdot|m'|}c_{i,|m'|}m'\otimes E_im+
f(|m'|,|m|)E_im'\otimes m\\
=&f(|m'|,|m|)(K'_im'\otimes E_im+E_im'\otimes m)\\
=&f(|m'|,|m|)\bar\Delta(E_i)(m'\otimes m).
\end{align*}
This finishes the proof.
\end{proof}

For any finite dimension weight $U_{v,t}$-module $M_1, M_2, M_3$,
we have maps $\mathcal{R}_{12}, \mathcal{R}_{23}:M_1\otimes M_2\otimes M_3\rightarrow M_3\otimes M_2\otimes M_1$
defined as $\mathcal{R}\otimes Id$ and $Id \otimes \mathcal{R}$, respectively.
We shall now verify that $\mathcal{R}$ satisfy the quantum Yang-Baxter equation
$$
\mathcal{R}_{12}\circ \mathcal{R}_{23}\circ \mathcal{R}_{12}
=\mathcal{R}_{23}\circ \mathcal{R}_{12}\circ \mathcal{R}_{23}.
$$
We will need the following lemma.

For $1\leq s, l\leq 3$, we define $\tilde{f}_{sl}$ on $M_1\otimes M_2\otimes M_3$
via $\tilde{f}_{sl}(m_1\otimes m_2\otimes m_3)=f(|m_s|,|m_l|)m_1\otimes m_2
\otimes m_3$. Let $\Theta^{op}=\sum_{\mu}\sum_{b\in \mathbf{B}_{\mu}}b^*\otimes b,
\Theta_{12}=\sum_{\mu}\sum_{b\in \mathbf{B}_{\mu}}b\otimes b^*\otimes 1,$ and
we define the other expressions in a similar way.
Letting $\Theta^{f}_{sl}=\Theta_{sl}\circ \tilde{f}_{sl}$,
we have the following identities for operators on $M_1\otimes M_2\otimes M_3$.
\begin{lem}\label{lem:Ridentity}
$(\mathrm{i})\ (\Delta\otimes 1)(\Theta^{op})\circ \tilde{f}_{31}
\circ \tilde{f}_{32}=\Theta^f_{31}\circ \Theta^f_{32}$.\\
$(\mathrm{ii})\ \tilde{f}_{31}\circ \tilde{f}_{32}
\circ \Theta_{12}=\Theta_{12}\circ \tilde{f}_{31}\circ \tilde{f}_{32}.$
\end{lem}
\begin{proof}
We shall give a detailed proof of $(\mathrm{i})$.
For any $m_1\in M_1, m_2\in M_2, m_3\in M_3$, by Lemma \ref{lem:delta}, we have
\begin{equation*}
\begin{split}
&(\Delta\otimes 1)(\Theta^{op})\circ \tilde{f}_{31}\circ \tilde{f}_{32}(m_1\otimes m_2\otimes m_3)\\
=&f(|m_3|,|m_1|)f(|m_3|,|m_2|)(\Delta\otimes 1)(\sum_{\mu}\sum_{b\in \mathbf{B}_{\mu}}b^*\otimes b)(m_1\otimes m_2\otimes m_3)\\
=&f(|m_3|,|m_1|)f(|m_3|,|m_2|)(\sum_{\mu, b\in \mathbf{B}_{\mu}}
  \sum_{0\leq \lambda \leq \mu, b_1\in \mathbf{B}_{\lambda},b_2\in \mathbf{B}_{\mu-\lambda}}
  (b^*,b_2b_1)_{\phi}b^*_2K_{\lambda}\otimes b^*_1\otimes b)(m_1\otimes m_2\otimes m_3)\\
=&f(|m_3|,|m_1|)f(|m_3|,|m_2|)(\sum_{\mu}
  \sum_{0\leq \lambda \leq \mu ,b_1\in \mathbf{B}_{\lambda},b_2\in \mathbf{B}_{\mu-\lambda}}
  b^*_2K_{\lambda}m_1\otimes b^*_1m_2\otimes b_2b_1m_3.
\end{split}
\end{equation*}
On the other hand,
\begin{equation*}
\begin{split}
&\Theta^f_{31}\circ \Theta^f_{32}(m_1\otimes m_2\otimes m_3)\\
=&f(|m_3|,|m_2|)\Theta^f_{31}(\sum_{\nu,b'\in \mathbf{B}_{\nu}}m_1\otimes b'^*m_2\otimes b'm_3)\\
=&f(|m_3|,|m_2|)\sum_{\nu, b'\in \mathbf{B}_{\nu}}\sum_{\zeta, b''\in \mathbf{B}_{\zeta}}
  f(|m_3|-\nu,|m_1|)b''^*m_1\otimes b'^*m_2\otimes b''b'm_3.
\end{split}
\end{equation*}

By Lemma \ref{lem:hopfpair} and \eqref{eq:fdef}, we have
$$K_{\nu}\cdot m_1=f(-\nu,|m_1|)m_1.$$
We get the second expression by replacing the variables $\lambda$ and $\mu-\lambda$
in the first expression with $\nu$ and $\zeta$, respectively.
This proves $(\mathrm{i})$.

Identity $(\mathrm{ii})$ follows directly from \eqref{eq:fprop}.
\end{proof}

We thus obtain the following crucial property of $\mathcal R$.
\begin{prop}\label{prop:braidrel}
For any finite dimension weight $U_{v,t}$-modules $M_1$, $M_2$, and $M_3$, we have
\[\mathcal{R}_{12}\circ \mathcal{R}_{23}\circ \mathcal{R}_{12}=
\mathcal{R}_{23}\circ\mathcal{R}_{12}\circ \mathcal{R}_{23}:
M_1\otimes M_2\otimes M_3\rightarrow M_3\otimes M_2\otimes M_1.\]
\end{prop}
\begin{proof}
This proposition follows from Lemma \ref{lem:Ridentity},
Proposition \ref{prop:Rmat} and the following identities.
\begin{align*}
P_{12}\circ\Theta^{f}_{12}=\Theta^{f}_{12}\circ P_{12},\quad
&P_{23}\circ\Theta^{f}_{23}=\Theta^{f}_{32}\circ P_{23},\quad
P_{23}\circ\Theta^{f}_{31}=\Theta^{f}_{21}\circ P_{23},\\
P_{23}\circ\Theta^{f}_{12}=\Theta^{f}_{13}\circ P_{23},\quad
&P_{12}\circ\Theta^{f}_{13}=\Theta^{f}_{23}\circ P_{12}.
\end{align*}
\end{proof}

\noindent\textbf{Remark.}
By specialization of $\mathcal{R}$-matrix, we recover the one for various quantum algebras in literatures.\\
(1) By setting $t=1$, the $\mathcal{R}$-matrix in Proposition \ref{prop:Rmat} degenerates into the one for
one-parameter quantum algebras\cite[Chapter 32]{Lusztigbook}.\\
(2) By setting $v=(rs^{-1})^{\frac{1}{2}}$ and $t=(rs)^{-\frac{1}{2}}$,
the $\mathcal{R}$-matrix in Proposition \ref{prop:Rmat} coincides with the one in \cite{BW}.

\section{Knot invariants associated to the two-parameter quantum algebra $U_{v,t}$}\label{sec:knotinvariant}

\subsection{Tangles}
Recall the definition of the tangle from \cite{Kassel}.
Let $[0]$ be the empty set and  $[n]=\{1,2, \cdots, n\}$ for any integer $n>0$.
We denote by $J$ the closed interval $[0, 1]$ and by $\mathbb{R}^2$ the real plane.
\begin{dfn}\cite[Section X.5]{Kassel}
Let $k$ and $l$ be nonnegative integers.
A tangle $L$ of type $(k, l)$ is the union of a finite number of
pairwise disjoint simple oriented polygonal arcs in $X=\mathbb{R}^2 \times J$
 such that the boundary $\partial L$ of $L$ satisfies the condition
 $$
 \partial L=L\cap (\mathbb{R}^2 \times \{0,1\})=([k]\times \{0\} \times \{0\})\cup ([l]\times \{0\} \times \{1\}).
 $$
\end{dfn}

For a tangle $L$ of type $(k,l)$, there exists two finite sequences
$s(L)$ and $b(L)$ consisting of $+$ and $-$ signs.
Let $s(L)=(\varepsilon_1, \cdots , \varepsilon_k)$ and $b(L)=(\eta_1, \cdots , \eta_l)$.
For $1\leq i\leq k$, let $\varepsilon_i=+$ (resp. $\varepsilon_i=-$)
if the point $(i, 0, 0)$ is an origin (resp. an endpoint) of $L$.
On the contrary, for $1\leq i\leq l$, let $\eta_i=+$ (resp. $\eta_i=-$)
if the point $(i, 0, 1)$ is an endpoint (resp. an origin) of $L$.

$$
\hctikz{\idsup{-1}{1}{0.5}{0.5}{0.5}}\ ,\ \ \hctikz{\idsdown{1}{1}{0.5}{0.5}{0.5}}\ ,\ \
\hctikz{\ccwcap{0.5}{-.1}{0.6}{.5}}\ ,\ \ \hctikz{\cwcap{0.5}{-.1}{0.6}{.5}}\ ,\ \
\hctikz{\cwcup{0.5}{-.1}{0.6}{.5}}\ ,\ \ \hctikz{\ccwcup{0.5}{-.1}{0.6}{.5}}\ ,\ \
\hctikz{\rcrossup{0}{0}{0.5}{.5}{0.5}}\ ,\ \ \hctikz{\lcrossup{0}{0}{0.5}{0.5}{0.5}}
$$
$$\mathrm{Figure}\ 1$$

From left to right, we denote the oriented tangles shown in Figure 1 by the symbols
$\uparrow, \downarrow, \curvearrowleft, \curvearrowright,
\undercurvearrowleft, \undercurvearrowright, X_+$ and $X_-$, respectively.

\noindent$\textbf{Example}$
(1) For the tangles $\uparrow$ and $\downarrow$,
we have
$$s(\uparrow)=(+), b(\uparrow)=(+),s(\downarrow)=(-) \ \rm{and}\  b(\downarrow)=(-).$$
(2)For the tangles $\curvearrowleft$,
we have $s(\curvearrowleft)=(-,+), b(\curvearrowleft)=\emptyset$.

The product $\circ$ and the tensor product $\otimes$ on the set of all isotopy types of tangles are defined as follows
$$
\hctikz{\standin{0}{.1}{.8}{.8}{L}} \circ \hctikz{\standin{0}{.1}{.8}{.8}{T}}=
\hctikz{\standin{0}{.1}{.8}{.8}{L} \standin{0}{-.8}{.8}{.8}{T}}\ ,
\quad
\hctikz{\standin{0}{.1}{.8}{.8}{L}} \otimes \hctikz{\standin{0}{.1}{.8}{.8}{T}}=
\hctikz{\standin{0}{.1}{.8}{.8}{L}}\ \hctikz{\standin{0}{.1}{.8}{.8}{T}}\ ,
$$
where $L$ and $T$ are tangles.
The product $L\circ T$ is well defined when $s(L)$ = $b(T)$.

\subsection{The strict monoidal category $(\rm{OTa},\otimes, \emptyset)$ of tangles}

Let ${\rm OTa}$ be a category. The objects of ${\rm OTa}$ are all finite sequences
consisting of $+$ and $-$ including  the empty sequence.
Given two finite sequences $\varepsilon=(\varepsilon_1, \cdots \varepsilon_k)$
and $\nu=(\nu_1, \cdots \nu_l)$, the oriented $(k,l)$-tangle $L$
represents a morphism from $\varepsilon$ to $\nu$
such that $s(L)=\varepsilon$ and $b(L)=\nu$.
The tensor product of the objects $\varepsilon$ and $\nu$
is the object $(\varepsilon, \nu)$.
The composition (resp. tensor product) of morphisms is the product (resp. tensor product) of tangles.
It is clear that $({\rm OTa}, \otimes, \emptyset)$
is a strict monoidal category\cite[Proposition XII.2.1]{Kassel}.

In what follows, we shall omit the $\otimes$ sign between morphisms
if there is no dangers of confusion.

\begin{thm}\label{thm:OTa}~\cite[Theorem 3.2]{Tur90}
The category ${\rm OTa}$ is generated by the morphisms $\undercurvearrowleft,\ \undercurvearrowright,
\ \curvearrowleft,\ \curvearrowright,\ X_+$ and
$\ X_-$ , and is presented by them together with the relations:
 \begin{itemize}
   \item[(i)] $(\curvearrowright\ \uparrow)\circ(\uparrow\ \undercurvearrowright)
=\ \uparrow\
=(\uparrow\ \curvearrowleft)\circ(\undercurvearrowleft\ \uparrow);$
\vspace{6pt}
\item[(ii)] $(\curvearrowleft\ \downarrow)\circ(\downarrow\ \undercurvearrowleft)
=\ \downarrow\
=(\downarrow\ \curvearrowright)\circ(\undercurvearrowright\ \downarrow);$
\vspace{6pt}
 \item[(iii)] $\ \ \ (\downarrow\ \downarrow\ \curvearrowright)
\circ (\downarrow\ \downarrow\ \uparrow\ \curvearrowright\ \downarrow)
\circ(\downarrow\ \downarrow\ X_{\pm}\ \downarrow\ \downarrow)
\circ(\downarrow\ \undercurvearrowright\ \uparrow\ \downarrow\ \downarrow)
\circ(\undercurvearrowright\ \downarrow\ \downarrow)\\
=
(\curvearrowleft\ \downarrow\ \downarrow)
\circ (\downarrow\ \curvearrowleft\ \uparrow\ \downarrow\ \downarrow)
\circ(\downarrow\ \downarrow\ X_{\pm}\ \downarrow\ \downarrow)
\circ(\downarrow\ \downarrow\ \uparrow\ \undercurvearrowleft\ \downarrow)
\circ(\downarrow\ \downarrow\ \undercurvearrowleft);
$
\vspace{6pt}
\item[(iv)] $
\ X_+\ \circ\ X_-\
=\ X_-\ \circ\ X_+\ \
=\ \uparrow\ \uparrow\ ;
$
\vspace{6pt}
\item[(v)] $
(X_+\ \uparrow)\circ (\uparrow\ X_+)\circ(X_+\ \uparrow)
=
(\uparrow \ X_+)\circ(X_+\ \uparrow)\circ(\uparrow\ X_+);
$
\vspace{6pt}
\item[(vi)] $
(\uparrow\ \curvearrowright)\circ (X_{\pm}\ \downarrow)\circ(\uparrow\ \undercurvearrowleft)
=\ \uparrow\ ;$
\vspace{6pt}
\item[(vii)] $
Y\circ T\ =\ \downarrow\ \uparrow\ ,\ \
T\circ Y\ =\ \uparrow\ \downarrow\ ,
$
\vspace{6pt}
\end{itemize}
where $Y=(\downarrow\ \uparrow\ \curvearrowright)
\circ(\downarrow\ X_+\ \downarrow)
\circ(\undercurvearrowright\ \uparrow\ \downarrow)$
and $T=(\curvearrowleft\ \uparrow\ \downarrow)
\circ(\downarrow\ X_-\ \downarrow)
\circ(\downarrow\ \uparrow\ \undercurvearrowleft).$
\end{thm}

\subsection{The strict monoidal category $(\rm{Mod},\otimes, \mathbb{Q}(v,t))$ of $U_{v,t}$-modules}
Let $\{M(i)\}_{i\in \{+,-\}}$ be a collection of $U_{v,t}$-modules
where $M(+)$ is a finite dimensional weight module
of the two parameter quantum algebra $U_{v,t}$ with the
highest weight $\lambda$ and $M(-)$ is its dual module.
To each finite sequence $j =(i_1, \cdots , i_n)$ with
$ i_1, \cdots , i_n \in \{+,-\}$ we associate the $U_{v,t}$-module
$$
M(j)=(\cdots ((M(i_1)\otimes_{U_{v,t}}M(i_2)) \otimes_{U_{v,t}}M(i_3))
\otimes\cdots ) \otimes_{U_{v,t}}M(i_n)).
$$
For the empty sequence $\emptyset$, we set $M(\emptyset)=\mathbb{Q}(v,t)$.

Consider the category ${\rm Mod}$ whose objects
are the pairs $(j , M(j))$ for all finite sequences
$j$ of elements of $\{+,-\}$.
The morphisms $(j, M(j))\rightarrow (j', M(j'))$ consist of
all $U_{v,t}$-linear homomorphisms $M(j)\rightarrow M(j').$
Composition of morphisms is the usual composition of homomorphisms.
For simplicity, we denote the object $(j , M(j))$ by $M(j)$.
The tensor product $\otimes$ is defined by setting $M(j)\otimes M(j')=M(j, j')$.

\subsection{Knot invariants}\label{subsec:knotinvariant}
Recall that $\mathrm{ev}, \mathrm{qtr}, \mathrm{coev}$, and
 $\mathrm{coqtr}$ are the maps defined in Lemma \ref{lem:fourmaps}.
Likewise, let $\mathcal{R}$ be the map defined in Proposition \ref{prop:Rmat}.

We denote the maps $\mathrm{id}_+, \mathrm{id}_-, \mathrm{ev}, \mathrm{qtr},
\mathrm{coev}, \mathrm{coqtr}, \mathcal{R}$ and $\mathcal{R}^{-1}$ by the symbols
$\uparrow, \downarrow, \curvearrowleft, \curvearrowright,
\undercurvearrowright, \undercurvearrowleft, X_+$ and $X_-$, respectively.
Then we have some substantial diagrammatic identities.
\begin{lem}\label{lem:relation12}
We have four equalities of diagrams
\[\hctikz{
\idsup{1}{.75}{.5}{.75}{1}\dcap{0}{.75}{.5}{.5}
\dcup{.5}{.25}{.5}{.5}\idsup{0}{0}{.5}{.75}{1}
}\quad =\quad  \hctikz{\idsup{0}{0}{1}{1.5}{1};
}\quad=\quad
\hctikz{
\idsup{0}{.75}{.5}{.75}{1}\dcap{.5}{.75}{.5}{.5}
\dcup{0}{.25}{.5}{.5}\idsup{1}{0}{.5}{.75}{1};
}\ ,
\]
% idsdown/idsup  represent arrow;   \draw (1.3,0)
\[\hctikz{
\idsdown{1}{.75}{.5}{.75}{1}\dcap{0}{.75}{.5}{.5}
\dcup{.5}{.25}{.5}{.5}\idsdown{0}{0}{.5}{.75}{1}
}\quad =\quad  \hctikz{\idsdown{0}{0}{1}{1.5}{1};
}\quad=\quad
\hctikz{
\idsdown{0}{.75}{.5}{.75}{1}\dcap{.5}{.75}{.5}{.5}
\dcup{0}{.25}{.5}{.5}\idsdown{1}{0}{.5}{.75}{1};
}\ .
\]
\end{lem}
\begin{proof}
The proofs of the four equalities are similar.
We show the first equality in detail.
In terms of morphisms, we wish to show
$$
(\mathrm{qtr} \otimes \mathrm{id}_+)\circ (\mathrm{id}_+ \otimes \mathrm{coev})=\mathrm{id}_+.
$$
Let $B$ be a homogeneous basis of $M(+)$ and $B^*$ the dual basis of $M(-)$.
Then for any $w_0\in B$, we have
\begin{align*}
({\rm qtr} \otimes {\rm id}_+)\circ ({\rm id}_+ \otimes {\rm coev})(w_0)
&=\sum_{w\in B}v^2_{|w|}({\rm qtr} \otimes {\rm id}_+)(w_0\otimes w^*\otimes w)\\
&=\sum_{w\in B}v^2_{|w|}v^2_{-|w_0|}w^*(w_0)w=w_0.
\end{align*}
\end{proof}
To distinguish $\hctikz{\rcrossdown{0.5}{0}{.5}{.5}}$ from $\hctikz{\rcrossup{0}{0}{0.5}{.5}{0.5}}$ ,
we let
$$
\mathcal{R}_{+,+}=\hctikz{\rcrossup{0}{0}{0.5}{.5}{0.5}}\ ,
\qquad \mathcal{R}_{-,-}=\hctikz{\rcrossdown{0.5}{0}{.5}{.5}}\ .
$$
\begin{lem}\label{lem:relation03}
We have four equalities of diagrams
\[(1)\quad
\hctikz{
\dcap{0}{.75}{1.5}{1}\dcap{.5}{.75}{.5}{.5}
\rcrossdown{1}{0}{.75}{.5}
\idsup{0}{0}{.5}{.75}{2};
}\quad = \quad
\hctikz{
\dcap{0}{.75}{1.5}{1}\dcap{.5}{.75}{.5}{.5}
\rcrossup{0}{0}{.75}{.5}
\idsdown{1}{0}{.5}{.75}{2};
}\ ,
\quad\quad(2)\quad
\hctikz{
\dcap{0}{.75}{1.5}{1}\dcap{.5}{.75}{.5}{.5}
\rcrossup{1}{0}{.75}{.5}
\idsdown{0}{0}{.5}{.75}{2};
}\quad = \quad
\hctikz{
\dcap{0}{.75}{1.5}{1}\dcap{.5}{.75}{.5}{.5}
\rcrossdown{0}{0}{.75}{.5}
\idsup{1}{0}{.5}{.75}{2};
}\ ,
\]

\[(3)\quad
\hctikz{
\dcap{0}{.75}{1.5}{1}\dcap{.5}{.75}{.5}{.5}
\lcrossdown{1}{0}{.75}{.5}
\idsup{0}{0}{.5}{.75}{2};
}\quad = \quad
\hctikz{
\dcap{0}{.75}{1.5}{1}\dcap{.5}{.75}{.5}{.5}
\lcrossup{0}{0}{.75}{.5}
\idsdown{1}{0}{.5}{.75}{2};
}\ ,
\quad\quad(4)\quad
\hctikz{
\dcap{0}{.75}{1.5}{1}\dcap{.5}{.75}{.5}{.5}
\lcrossup{1}{0}{.75}{.5}
\idsdown{0}{0}{.5}{.75}{2};
}\quad = \quad
\hctikz{
\dcap{0}{.75}{1.5}{1}\dcap{.5}{.75}{.5}{.5}
\lcrossdown{0}{0}{.75}{.5}
\idsup{1}{0}{.5}{.75}{2};
}\ .
\]
\end{lem}
\begin{proof}
The proofs of $(1)-(4)$ are similar.
We shall only show (1) in detail.
It is equivalent to show that (1) holds for the following equality
$$\varphi=\psi,$$
where
\begin{align*}
\varphi=&\rm{qtr} \circ (\rm{id}_+\otimes \rm{qtr} \otimes \rm{id}_-) \circ
(\rm{id}_+\otimes\rm{id}_+\otimes \mathcal{R}_{-,-}),\\
\psi=&\rm{qtr} \circ (\rm{id}_+\otimes \rm{qtr} \otimes \rm{id}_-) \circ
(\mathcal{R}_{+,+}\otimes \rm{id}_-\otimes\rm{id}_-).
\end{align*}

Let $m_1$, $m_2\in M(+)$ and $m_3^*$, $m_4^*\in M(-)$.
On the left hand side, by Lemma \ref{lem:S}, we have
\begin{align*}
&\varphi(m_1\otimes m_2\otimes m_3^*\otimes m_4^*)\\
=&\sum_{\nu}\sum_{b\in \mathbf{B}_{\nu}}v^2_{-|m_1|-|m_2|}
f(|m_4^*|,|m_3^*|)(b\cdot m_4^*)(m_2)(b^*\cdot m_3^*)(m_1)\\
=&\sum_{\nu}\sum_{b\in \mathbf{B}_{\nu}}v^2_{-|m_1|-|m_2|}
f(|m_4^*|,|m_3^*|)v^{\nu\cdot (|m_2|-|m_1|-\nu)}c_{-\nu, |m_1|+|m_2|}\\
&\qquad \qquad m_4^*(\sigma^-(b)m_2)m_3^*(\sigma^+(b^*)m_1).
\end{align*}
On the right hand side, by Lemma \ref{lem:hopfpairprop}, we rewrite
the presentation of $\Theta$ in the basis $\sigma^-(\mathbf{B})$ such as
$$
\Theta=\sum_{\nu}\sum_{b\in \mathbf{B}_{\nu}}\sigma^-(b)\otimes \sigma^+(b^*).
$$
Then, we have
\begin{align*}
&\psi(m_1\otimes m_2\otimes m_3^*\otimes m_4^*)\\
=&\sum_{\nu}\sum_{b\in \mathbf{B}_{\nu}}v^2_{-|m_1|-|m_2|} f(|m_2|,|m_1|)
m_4^*(\sigma^-(b)m_2)m_3^*(\sigma^+(b^*)m_1).
\end{align*}
It is enough to show that
\begin{equation*}
f(|m_2|,|m_1|)=f(|m_4^*|,|m_3^*|)
v^{\nu\cdot (|m_2|-|m_1|-\nu)}c_{-\nu, |m_1|+|m_2|}.
\end{equation*}
This equality holds for $|m_1|=-|m_3^*|-\nu$ and $|m_2|=-|m_4^*|+\nu$.
\end{proof}

By this lemma, we have the following corollary.
\begin{cor}\label{cor:relation3}
We have four equalities of diagrams
\[
\hctikz{
\rcrossdown{0}{0}{1}{1};
}\quad =\quad
\hctikz{
\dcap{0}{.75}{1.5}{1}\dcap{.5}{.75}{.5}{.5}\idsdown{2}{0}{.5}{1.5}{2}
\rcrossup{1}{0}{.75}{.5}
\idsdown{0}{-.75}{.5}{1.5}{2}\dcup{1.5}{-.5}{.5}{.5}\dcup{1}{-1}{1.5}{1};
}\quad =\quad
\hctikz{
\dcap{1}{.75}{1.5}{1}\dcap{1.5}{.75}{.5}{.5}\idsdown{0}{0}{.5}{1.5}{2}
\rcrossup{1}{0}{.75}{.5}
\idsdown{2}{-.75}{.5}{1.5}{2}\dcup{.5}{-.5}{.5}{.5}\dcup{0}{-1}{1.5}{1};
}\ ,\]

\[
\hctikz{
\lcrossdown{0}{0}{1}{1};
}\quad =\quad
\hctikz{
\dcap{0}{.75}{1.5}{1}\dcap{.5}{.75}{.5}{.5}\idsdown{2}{0}{.5}{1.5}{2}
\lcrossup{1}{0}{.75}{.5}
\idsdown{0}{-.75}{.5}{1.5}{2}\dcup{1.5}{-.5}{.5}{.5}\dcup{1}{-1}{1.5}{1};
}\quad =\quad
\hctikz{
\dcap{1}{.75}{1.5}{1}\dcap{1.5}{.75}{.5}{.5}\idsdown{0}{0}{.5}{1.5}{2}
\lcrossup{1}{0}{.75}{.5}
\idsdown{2}{-.75}{.5}{1.5}{2}\dcup{.5}{-.5}{.5}{.5}\dcup{0}{-1}{1.5}{1};
}\ .\]
\end{cor}

\begin{lem}\label{lem:relation06}
 We have two equalities of diagrams
\[
(f^{-1}(\lambda,\lambda)v^2_{\lambda})^{-1}\
\hctikz{
\ids{0}{.75}{.5}{.75}{1}\dcap{.5}{.75}{.5}{.5}
\lcrossup{0}{0}{.75}{.5}\idsdown{1}{0}{.5}{.75}{1}
\ids{0}{-.75}{.5}{.75}{1}\dcup{.5}{-.5}{.5}{.5};
}\quad =\quad\hctikz{\idsup{0}{0}{1}{1.5}{1};
}\quad=\quad f^{-1}(\lambda,\lambda)v^2_{\lambda}\
\hctikz{
\ids{0}{.75}{.5}{.75}{1}\dcap{.5}{.75}{.5}{.5}
\rcrossup{0}{0}{.75}{.5}\idsdown{1}{0}{.5}{.75}{1}
\ids{0}{-.75}{.5}{.75}{1}\dcup{.5}{-.5}{.5}{.5};
}\ ,
\]
where $\lambda$ is the highest weight of $M(+)$.
\end{lem}
\begin{proof}
We denote
\[\varphi=\hctikz{
	\ids{0}{.75}{.5}{.75}{1}\dcap{.5}{.75}{.5}{.5}
	\rcrossup{0}{0}{.75}{.5}\idsdown{1}{0}{.5}{.75}{1}
	\ids{0}{-.75}{.5}{.75}{1}\dcup{.5}{-.5}{.5}{.5};
}=(\rm{id}_+\otimes \qtr)\circ (\cR_{+,+}\otimes \rm{id}_+)\circ (\rm{id}_+\otimes \coqtr),\]
\[\psi=\hctikz{
	\ids{0}{.75}{.5}{.75}{1}\dcap{.5}{.75}{.5}{.5}
\lcrossup{0}{0}{.75}{.5}\idsdown{1}{0}{.5}{.75}{1}
\ids{0}{-.75}{.5}{.75}{1}\dcup{.5}{-.5}{.5}{.5};
}=(\rm{id}_+\otimes \qtr)\circ (\cR^{-1}_{+,+}\otimes \rm{id}_+)\circ (\rm{id}_+\otimes \coqtr).\]

Since $\varphi$ and $\psi$ are $U_{v,t}$-module homomorphisms from $M(+)$ to $M(+)$,
both $\varphi$ and $\psi$ must be a multiple
of the identity which is completely determined by
the image of an extremal weight vector.

Let $m_\lambda,\ m_{-\lambda} \in M(+)$
be nonzero highest-weight and lowest-weight vectors. We have
\begin{align*}
\varphi(m_{\lambda})&=(\mathrm{id}_+\otimes \mathrm{qtr})\circ (\mathcal{R}_{+,+}\otimes \mathrm{id}_+)
\sum_{w\in B}m_{\lambda}\otimes w\otimes w^*\\
&=(\mathrm{id}_+\otimes \mathrm{qtr})
(\sum_{w\in B}f(|w|,\lambda)w\otimes m_\lambda\otimes w^*)
= f(\lambda,\lambda)v^2_{-\lambda}m_\lambda.
\end{align*}
Thus $\mathrm{id}_+\ =\ f^{-1}(\lambda,\lambda)v^2_{\lambda}\ \varphi$.

By Corollary \ref{cor:bartheta}, we have
$$
\mathcal{R}^{-1}_{+,+}=P\circ \tilde{f}^{-1} \circ \bar{\Theta}.
$$
By Lemma \ref{lem:barquasiR}, we compute
\begin{align*}
\psi(m_{-\lambda})&=(\mathrm{id}_+\otimes \mathrm{qtr})\circ (\mathcal{R}^{-1}_{+,+}\otimes \mathrm{id}_+)
\sum_{w\in B}m_{-\lambda}\otimes w\otimes w^*\\
&=(\mathrm{id}_+\otimes \mathrm{qtr})
\sum_{w\in B} f^{-1}(-\lambda,|w|)w\otimes m_{-\lambda}\otimes w^*
 =f^{-1}(-\lambda,-\lambda)v^2_{\lambda} m_{-\lambda}.
\end{align*}
Then we have $\mathrm{id}_+=(f^{-1}(\lambda,\lambda)v^2_{\lambda})^{-1}\ \psi$ following \eqref{eq:fprop}.

This finishes the proof.
\end{proof}

\begin{lem}\label{lem:relation07}
We have two equalities of diagrams
\begin{equation*}\tag{1}
\hctikz{
\NESE{0}{0}{1}{1};
}\quad = \quad
\hctikz{
\idsdown{0}{0}{.5}{1.5}{1}\ids{.5}{.75}{.5}{.75}{1}\dcap{1}{.75}{.5}{.5}
\lcrossup{.5}{0}{.75}{.5}\idsdown{1.5}{-.75}{.5}{1.5}{1}\ids{1}{-.75}{.5}{.75}{1}
\dcup{0}{-.5}{.5}{.5};
},
\end{equation*}
\begin{equation*}\tag{2}
\hctikz{
\NWSW{0}{0}{1}{1};
}\quad = \quad
\hctikz{
\ids{1}{.75}{.5}{.75}{1}\idsdown{1.5}{0}{.5}{1.5}{1}\dcap{0}{.75}{.5}{.5}
\rcrossup{.5}{0}{.75}{.5}\idsdown{0}{-.75}{.5}{1.5}{1}\ids{.5}{-.75}{.5}{.75}{1}
\dcup{1}{-.5}{.5}{.5};
}.
\end{equation*}
\end{lem}

\begin{proof}
The proofs of (1) and (2) are similar. We shall only show (1) in detail.

Denoting $\ \hctikz{\NESE{0}{0}{0.5}{0.5}}$  by $\mathcal{R}_{+,-}$, (1) is equivalent to
\begin{equation*}%\label{eqcalRphi}
\mathcal{R}_{+,-}=\varphi,
\end{equation*}
where $\varphi=(\mathrm{id}_- \otimes \mathrm{id}_+\otimes \mathrm{qtr})\circ
(\mathrm{id}_-\otimes \mathcal{R}^{-1}_{+,+}\otimes \mathrm{id}_-)\circ
(\mathrm{coev}\otimes\mathrm{id}_+\otimes \mathrm{id}_-)$.

For any $m_1\in M(+)$ and $m^*_2\in M(-)$, we have
\begin{equation}\label{eq:R+-}
\mathcal{R}_{+,-}(m_1\otimes m^*_2)=f(|m^*_2|,|m_1|)
\sum_{\nu}\sum_{b\in \mathbf{B}_{\nu}}bm^*_2\otimes b^*m_1.
\end{equation}

By Lemma \ref{lem:hopfpairprop},
we have the representation of $\bar{\Theta}$ in the basis $\sigma(\mathbf{B})$ such as
$$
\bar{\Theta}=\sum_{\nu}(-1)^{\mathrm{tr}\nu}v^{\frac{\nu\cdot\nu}2}v_{-\nu}
\sum_{b\in \mathbf{B}_\nu} \sigma^-(b)\otimes b^*.
$$
Then we compute
\begin{equation}\label{eq:varphi12}
\begin{split}
&(\mathrm{id}_-\otimes \mathcal{R}^{-1}_{+,+}\otimes \mathrm{id}_-)\circ
(\mathrm{coev}\otimes\mathrm{id}_+\otimes \mathrm{id}_-)(m_1\otimes m^*_2)\\
=&\sum_{\nu}(-1)^{\mathrm{tr}\nu}v^{\frac{\nu\cdot\nu}2}v_{-\nu}
\sum_{b\in \mathbf{B}_\nu}\sum_{w\in B}v^2_{|w|}f^{-1}(|w|-\nu,|m_1|+\nu)\\
&\ w^*\otimes b^*m_1\otimes \sigma^-(b)w\otimes m^*_2.
\end{split}
\end{equation}
By Lemma \ref{lem:S}, we have
\begin{equation}\label{eq:varphi3}
\begin{split}
&(\mathrm{id}_- \otimes \mathrm{id}_+\otimes \mathrm{qtr})
(\sum_{w\in B}v^2_{|w|}f^{-1}(|w|-\nu,|m_1|+\nu)
w^*\otimes b^*m_1\otimes \sigma^-(b) w\otimes m^*_2)\\
=&\sum_{w\in B}v^2_{\nu}f^{-1}(|w|-\nu,|m_1|+\nu)
m^*_2(\sigma(b)w) w^*\otimes b^*m_1\\
=&\sum_{w\in B}(-1)^{{\rm tr}\nu}v^{\frac{\nu\cdot \nu}{2}}v_{-\nu}
v^{-\nu \cdot|w|}c_{\nu, |w|}v^2_{\nu}f^{-1}(|w|-\nu,|m_1|+\nu)
(bm^*_2)(w)w^*\otimes b^*m_1\\
=&(-1)^{{\rm tr}\nu}v^{\frac{\nu\cdot \nu}{2}}v_{\nu}
v^{\nu \cdot(|m^*_2|-\nu)}c_{\nu, -|m^*_2|}f(|m^*_2|,|m_1|+\nu)
bm^*_2\otimes b^*m_1.
\end{split}
\end{equation}
By \eqref{eq:fprop}, \eqref{eq:varphi12} and \eqref{eq:varphi3}, we have
\begin{equation}\label{eq:varphi}
\varphi(m_1\otimes m^*_2)=f(|m^*_2|,|m_1|)
\sum_{\nu}\sum_{b\in \mathbf{B}_{\nu}}bm^*_2\otimes b^*m_1.
\end{equation}
By \eqref{eq:R+-} and \eqref{eq:varphi}, we have
$$
\mathcal{R}_{+,-}(m_1\otimes m^*_2)=\varphi(m_1\otimes m^*_2).
$$
This finishes the proof.
\end{proof}
By this lemma, we have the following corollary.
\begin{cor}\label{cor:relation7}
We have two equalities of diagrams
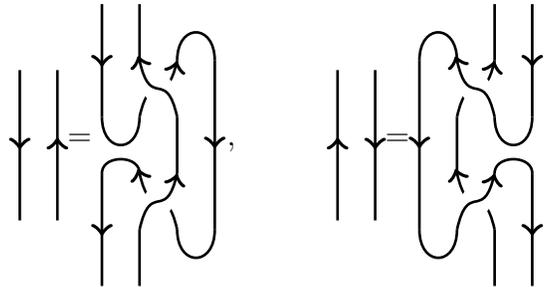
\[\quad\hctikz{
\idsdown{0}{0}{.5}{2}{1}
\idsup{.5}{0}{.5}{2}{1};
}=\hctikz{
\ids{.5}{2.25}{.5}{.75}{1}\idsdown{0}{1.5}{.5}{1.5}{1}\dcap{1}{2.25}{.5}{.5}
\lcrossup{.5}{1.5}{.75}{.5}
\dcup{0}{1}{.5}{.5}
\idsdown{1.5}{0}{.5}{2.25}{1}\ids{1}{.75}{.5}{.75}{1}\dcap{0}{.75}{.5}{.25}
\rcrossup{.5}{0}{.75}{.5}\idsdown{0}{-.75}{.5}{1.5}{1}\ids{.5}{-.75}{.5}{.75}{1}
\dcup{1}{-.5}{.5}{.5};
}\ , \qquad
\quad\hctikz{
\idsup{0}{0}{.5}{2}{1}
\idsdown{.5}{0}{.5}{2}{1};
}=\hctikz{
\ids{1}{2.25}{.5}{.75}{1}\idsdown{1.5}{1.5}{.5}{1.5}{1}\dcap{0}{2.25}{.5}{.5}
\lcrossup{.5}{1.5}{.75}{.5}
\dcup{1}{1}{.5}{.5}
\idsdown{0}{0}{.5}{2.25}{1}\ids{.5}{.75}{.5}{.75}{1}\dcap{1}{.75}{.5}{.25}
\rcrossup{.5}{0}{.75}{.5}\idsdown{1.5}{-.75}{.5}{1.5}{1}\ids{1}{-.75}{.5}{.75}{1}
\dcup{0}{-.5}{.5}{.5};
}\ .
\]
\end{cor}

\begin{thm}\label{thm:knot invariant}
There exists a strict tensor functor $\mathcal{T}$
from the tangle category $(\rm{OTa}, \otimes, \emptyset)$ to
the $U_{v,t}$-module category $(\rm{Mod}, \otimes, \mathbb{Q}(v,t))$
such that $\mathcal{T}((+)) =M(+), \mathcal{T}((-)) =M(-)$, and
\begin{equation*}
\begin{split}
&\mathcal{T}(X_{+})=(f(\lambda,\lambda)v_{-\lambda}^2)^{-1}\mathcal{R},
\quad \mathcal{T}(\undercurvearrowright)=\mathrm{coev},
\quad \mathcal{T}(\undercurvearrowleft)=\mathrm{coqtr},\\
&\mathcal{T}(X_{-})=f(\lambda,\lambda)v_{-\lambda}^2\mathcal{R}^{-1},
\quad \ \ \mathcal{T}(\curvearrowright)=\mathrm{qtr},
\quad\ \ \mathcal{T}( \curvearrowleft)=\mathrm{ev},
\end{split}
\end{equation*}
where $\lambda$ is the highest weight of $M(+)$.
\end{thm}
\begin{proof}
To prove the proposition, it is enough to show that the evaluations of $\mathcal{T}$
at both sides of the relations of Theorem \ref{thm:OTa} coincide.

The relations {\rm (i), (ii), (iii), (vi)} and {\rm(vii)} follow from
Lemma \ref{lem:relation12}, Corollary \ref{cor:relation3}, Lemma \ref{lem:relation06} and Corollary \ref{cor:relation7}, respectively.
The relations {\rm (iv)} and {\rm (v)} hold for the properties of $\mathcal{R}$.
\end{proof}

\noindent$\textbf{Remark}\quad $ Given a tangle $L$ of type $(n,n)$ for any $n \in \mathbb{N}$,
we get the closure $\tilde{L}$ of $L$ by connecting
the origin and the endpoint one by one with no intersection.
For example, see Figure 2.
Furthermore, the evaluation of the functor $\mathcal{T}$ on $\tilde{L}$
is an endomorphism of ground field $\mathbb{Q}(v,t)$.
Therefore, $\mathcal{T}(\tilde{L})(1)$ is a binary polynomial with the parameters $v$ and $t$.
That is called the quantum knot invariant.
We expect this is a refinement of the knot invariant associated to the one-parameter case.

\begin{center}
 L=\quad
\hctikz{
\rcrossup{1}{0}{.75}{.5}
}\ ,
\qquad\qquad
$\tilde{L}$=\quad
\hctikz{
\dcap{1}{.75}{1.5}{1}\dcap{1.5}{.75}{.5}{.5}
\rcrossup{1}{0}{.75}{.5}\idsdown{2}{0}{.5}{1}{2}
\dcup{1.5}{-.5}{.5}{.5}\dcup{1}{-1}{1.5}{1}
}\ .
\end{center}
\vspace{-7pt}
$$\mathrm{Figure}\ 2$$

%-----------------------------------------------------------------------------------------------------------


\begin{thebibliography}{DDPW08}\frenchspacing
\bibitem[BW]{BW}
G. Benkart and S. Witherspoon, \emph{Two-parameter quantum groups and Drinfel'd doubles}. Algebras \& Representation Theory, 2004, 7(3):261-286.

\bibitem[C]{C}
S. Clark, \emph{Odd knot invariants from quantum covering groups}. Algebraic \& Geometric Topology, 2017, 17(5):2961-3005.

\bibitem[CFYW]{CFYW}
S.~Clark, Z.~Fan, and Y.~Li, elt, \emph{Quantum supergroups III. Twistors.} Communications
in Mathematical Physics, 2014, 332(1): 415-436.

\bibitem[FL]{FL}
Z.~Fan and Y.~Li, \emph{Two-parameter quantum algebras, canonical bases and categorifications}. International Mathematics Research Notices, 2015, 2015(16):7016-7062.

\bibitem[FX]{FX}
Z.~Fan and J.~Xing, \emph{Drinfeld double of deformed quantum algebras}. Journal of Algebra, 2019, 534:358-383.

\bibitem[HPR]{HPR}
N.~Hu, Y.~Pei, and M.~Rosso, \emph{Multi-parameter quantum groups and quantum shuffles. I.}
Quantum Affine Algebras, Extended Affine Lie Algebras, and their Applications, 145-71.
Contemp. Math. 506. Providence, RI: American Mathematical Society, 2010.

\bibitem[Jan]{Ja}
J.C.~Jantzen, \emph{Lectures on quantum groups}. American Mathematical Society, Providence, 1996.

\bibitem[Kas]{Kassel}
C.~Kassel, {\em Quantum groups}. Springer-Verlag, New York,  1995.

\bibitem[KM]{KM}
L.H.~Kauffman and V.O.~Manturov, \emph{Graphical constructions for the $sl(3)$, $C_2$ and $G_2$ invariants for virtual knots, virtual braids and free knots.} Journal of Knot Theory and Its Ramifications, 2015, 24(06): 1550031.

\bibitem[K11]{Kau11}
L. H.~Kauffman, \emph{Remarks on Khovanov homology and the Potts model.} Perspectives in Analysis, Geometry, and Topology, 2011:237-262.

\bibitem[K16]{Kau}
L. H.~Kauffman, \emph{Knot logic and topological quantum computing with Majorana fermions.} Logic and Algebraic Structures in Quantum Computing, 2016:223-336.

\bibitem[L]{Lusztigbook}
G.~Lusztig, \emph{Introduction to quantum groups}. Modern Birkh$\ddot{a}$user Classics, Birkh$\ddot{a}$user$/$Springer, New York, 2010.


\bibitem[NSS]{N}
C.~Nayak, S. H.~Simon and A.~Stern, elt, \emph{Non-Abelian anyons and topological quantum computation.} Reviews of Modern Physics, 2008, 80:1083-1159.

\bibitem[RT]{RT90}
N. Y.~Reshetikhin and V. G.~Turaev, \emph{Ribbon graphs and their invariants derived from quantum groups.} Communications in Mathematical Physics, 1990, 127:1-26.

\bibitem[Tur90]{Tur90}
V G.~Turaev, \emph{Operator invariants of tangles, and R-matrices.} Mathematics of the USSR-Izvestiya, 1990, 35(2):411-444.

\bibitem[W]{W}
E.~Witten, \emph{Quantum field theory and the jones polynomial.} Communications in Mathematical Physics, 1989, 121(3): 351-399.

\bibitem[X97]{Xiao1}
J.~Xiao, \emph{Drinfeld double and Ringel-Green Theory of Hall algebras}. Journal of Algebra, 1997, 190(1):100-144.

\bibitem[ZGB]{ZGB}
R.~Zhang, M.~Gould and  A.~Bracken, \emph{Quantum group invariants and link polynomials.} Communications in Mathematical Physics, 1991, 137(1): 13-27.

\end{thebibliography}
\end{document}